\documentclass{amsart}

\usepackage[latin1]{inputenc}
\usepackage{amsfonts}
\usepackage{amsmath}
\usepackage{amsthm}
\usepackage{amssymb}
\usepackage{latexsym}
\usepackage{enumerate}

\newtheorem{theorem}{Theorem}[section]
\newtheorem{lemma}[theorem]{Lemma}
\newtheorem{proposition}[theorem]{Proposition}

\newtheorem{claim}{Claim}

\newtheorem{definition}[theorem]{Definition}

\numberwithin{equation}{section}

\newtheoremstyle{TheoremNum}
  {\topsep}{\topsep}              
  {\itshape}                      
  {}                              
  {\bfseries}                     
  {.}                             
  { }                             
  {\thmname{#1}\thmnote{ \bfseries #3}}
\theoremstyle{TheoremNum}
\newtheorem{reptheorem}{Theorem}

\DeclareMathOperator\supp{supp}

\DeclareMathOperator\ext{ext}
\begin{document}

\newcommand{\cc}{\mathfrak{c}}
\newcommand{\N}{\mathbb{N}}
\newcommand{\R}{\mathbb{R}}
\newcommand{\C}{\mathbb{C}}
\newcommand{\PP}{\mathbb{P}}
\newcommand{\A}{\mathcal{A}}
\newcommand{\B}{\mathcal{B}}
\newcommand{\D}{\mathcal{D}}
\newcommand{\E}{\mathcal{E}}
\newcommand{\F}{\mathcal{F}}
\newcommand{\G}{\mathcal{G}}
\newcommand{\I}{\mathcal{I}}
\newcommand{\U}{\mathcal{U}}

\title [Indecomposable Banach spaces] 
{There is no bound on sizes of indecomposable  Banach spaces}

\author{Piotr Koszmider}
\address{Institute of Mathematics, Polish Academy of Sciences,
ul. \'Sniadeckich 8,  00-656 Warszawa, Poland}
\email{\texttt{piotr.koszmider@impan.pl}}
\thanks{This research  was partially supported by   grant
PVE Ci\^encia sem Fronteiras - CNPq (406239/2013-4).}

\author{Saharon Shelah}
\address{Department of Mathematics, The Hebrew University of Jerusalem, 90194 Jerusalem,
Israel, and
Rutgers University, Piscataway, NJ 08854-8019, USA}
\email{\texttt{shelah@math.huji.ac.il}}

\author{Micha{\l} \'Swi\c etek}
\address{Faculty of Mathematics and Computer Science, Jagiellonian University,
{\L}ojasiewicza 4, 30-348 Krak\'ow, Poland}
\email{\texttt{michal.swietek@uj.edu.pl}}
\thanks{The contribution of the third named author
was a part of the doctoral internship  project conducted at the Warsaw Center of Mathematics
and Computer Science in the academic year 2014/15.}

\subjclass[2010]{46B25, 46B28, 47L10, 47B38, 03E35, 54D30,  06A40}

\begin{abstract} Assuming the generalized
continuum hypothesis we construct arbitrarily big indecomposable Banach spaces.
i.e., such that whenever they are decomposed as $X\oplus Y$, then one of the closed
subspaces $X$ or $Y$ must be finite dimensional. It requires alternative techniques
compared to those which were initiated by Gowers and Maurey or Argyros with the coauthors. This is
because  hereditarily indecomposable Banach spaces 
 always embed into $\ell_\infty$ and so their density and cardinality is bounded by the continuum
and because dual Banach spaces of densities bigger than continuum are decomposable by a result due
to Heinrich and Mankiewicz.

 The obtained Banach spaces are of the form $C(K)$ for some compact connected
Hausdorff space and have few operators in the sense that every linear bounded operator $T$
on $C(K)$  for every $f\in C(K)$ satisfies $T(f)=gf+S(f)$ where $g\in C(K)$ and $S$
is weakly compact or equivalently strictly singular.  In particular,
 the spaces carry the structure of a Banach algebra  and in the complex case
even the structure of a $C^*$-algebra.
\end{abstract}

\maketitle

\section{Introduction}

The research in the classical period of the isomorphic theory of Banach spaces led to
 questions of Lindenstrauss (\cite{lindenstrauss})
and Johnson and Lindenstrauss (\cite{johnsonlinden}), respectively, which can be phrased
as follows:
\vskip 6pt
\begin{enumerate}
\item[(A)]  {\it Is it true that every infinite dimensional Banach space has a complemented infinite
dimensional and infinite codimensional subspace?}
\vskip 6pt
\item[(B)] {\it Is it true that every infinite dimensional Banach space has a complemented infinite
dimensional subspace of density $\leq$ continuum?}
\end{enumerate}
\vskip 6pt
 Recall that a linear
closed subspace $Y$ of a Banach space $X$ is complemented in $X$ if there is 
another closed linear subspace $Z\subseteq X$ such that $Y\cap Z=\{0\}$
and $Y+Z=X$. $Y$ is complemented
in $X$ if and only if there is a bounded linear projection from $X$ onto $Y$
 (\cite{semadeni}).

The first, spectacular negative solution to question (A) (such spaces
are called indecomposable Banach spaces) was obtained by Gowers and Maurey in  \cite{gowers},
where they constructed an infinite dimensional separable Banach space which has even a stronger property 
of being hereditarily indecomposable, i.e.,  each of its 
infinite dimensional closed subspaces is indecomposable. Being hereditary
indecomposable is tightly related to having few operators in the sense that
 every operator on the space  is a 
strictly singular perturbation of a multiple of identity (see \cite{ferenczi} 
for exact description of the relation in both the real and the complex case). 
Every operator on a hereditarily indecomposable 
Banach spaces may even be a  compact perturbation of a multiple of identity
as recently proved by Argyros and Haydon (\cite{argyros-haydon}).
Many constructions of indecomposable Banach spaces followed
the paper of Gowers and Maurey, however most of
them, including nonseparable ones, were hereditarily indecomposable, which as proved
e.g., in \cite{argyros-tolias} or \cite{plichko-yost}, must embed in $\ell_\infty$
which limits their density character or cardinality to the continuum. This led
to the following question of S. Argyros:
\vskip 6pt
\begin{enumerate}
\item[(C)] {\it Is there an upper bound for densities of indecomposable Banach spaces?}
\end{enumerate}
\vskip 6pt
Assuming various additional 
properties of a Banach space  the positive answer
to   question (B) and so to question (C) has been obtained
by many authors, for a survey of this type of results see
\cite{plichko-yost}. As many hereditarily indecomposable spaces are dual Banach 
spaces (see \cite{argyros-tolias}) most relevant for us is the result of Heinrich and Mankiewicz \cite{mankiewicz},
which says that dual Banach spaces of density bigger than continuum are decomposable.
Also several new upper bounds for densities of Banach spaces
with some rigidity concerning basic sequences were recently obtained by P. Dodos, J. Lopez-Abad, S. Todorcevic 
(\cite{dodos},  \cite{jordilarge}, \cite{positional}).

In the meantime a different kind of indecomposable Banach spaces was introduced in \cite{few}
by the first named author, namely,  spaces of continuous functions\footnote{By $C(K)$ we understand the Banach 
space of all  real-valued continuous functions on a compact Hausdorff space $K$ with the
supremum norm.} with few operators, 
or with few$^*$ operators
in the sense of the following:

\begin{definition}\label{multiplier}
  Let $K$ be a compact Hausdorff space and let $T:C(K)\rightarrow C(K)$ be a 
bounded linear operator on $C(K)$. 
  \begin{enumerate}
	\item  $T$ is called a weak multiplier if $T^*=gI+S$ where $g\colon K \rightarrow \R$ is a Borel bounded function 
	  and $S$ is a weakly compact operator on $C(K)^*$,
	\item  $T$ is called a weak multiplication if $T=gI+S$ where $g\in C(K)$
	  and $S$ is a weakly compact operator on $C(K)$,
\item The Banach space $C(K)$ has few  operators (few$^*$ operators) if every linear bounded operator
on $C(K)$ is a weak multiplication (weak multiplier),
\item A point $x\in K$ is called a butterfly point if and only if there are disjoint open
$U, V\subseteq K$ such that $\overline U\cap \overline V=\{x\}$.
  \end{enumerate}
\end{definition}
  
We have the following:

\begin{theorem}[2.5., 2.7., 2.8 \cite{few}, 13 \cite{fewsur}]\label{few-theorems}
Suppose that $K$ is compact  Hausdorff. 
\begin{itemize}
\item If $C(K)$ has few operators and $K$ is connected, then $C(K)$ is indecomposable,
\item If $C(K)$ has few$^*$ operators and $K\setminus F$ is connected for any
finite $F\subseteq K$, then $C(K)$ is indecomposable,
\item If $C(K)$ has few$^*$ operators and $K$ has no butterfly points,  then 
$C(K)$ has few operators.
\end{itemize}
\end{theorem}

The first constructions of an indecomposable Banach space as above 
(with few$^*$ operators in ZFC and with few operators under CH, both $K$s separable) of density
continuum  appeared
in \cite{few} and some improvements followed, among others, in \cite{plebanek} (with
few operators in ZFC for $K$ nonseparable) and  \cite{iryna} (with
few operators in ZFC for $K$ separable), for
a survey see \cite{fewsur}.
In \cite{big} and \cite{grande} the first consistent examples of 
Banach spaces  giving the negative answer
to question (B) and (A) respectively were presented.
They were  Banach spaces of the form $C(K)$ with few operators,
 however the (forcing) method was limited to the density $2^{\omega_1}$.
Note that by the classification of separable Banach spaces of
the form $C(K)$ due to Milutin, Bessaga and Pe\l czy\'nski (\cite{semadeni})
indecomposable $C(K)$s must  be nonseparable.
 On the other hand it is consistently possible to obtain  indecomposable 
$C(K)$s, with few operators of densities strictly smaller than continuum (\cite{rogerio}). 
It should be also added that the classes of strictly singular and
weakly compact operators coincide for $C(K)$ spaces (\cite{pelczynski}).

The main result of this paper is to give the negative answer to question (C)
and strengthen (compared to \cite{big} and \cite{grande}) the negative answer to question (B)
and provide new examples relevant to question (A) by proving:

\begin{theorem}\label{theorem1}  Assume the generalized continuum hypothesis.
 For every cardinal $\kappa$ there is
an indecomposable  Banach space of density bigger than $\kappa$. In particular
 it has no  infinite dimensional complemented
subspace of density smaller than $\kappa$. The spaces are (real Banach
algebras) of the form $C(K)$ with few operators where $K$ is compact Hausdorff and connected.
\end{theorem}
\begin{proof} Use Theorems \ref{few-theorems}, \ref{theorem2}
and \ref{theorem3}.
\end{proof}
The methods of
the paper consist of a fusion of  the techniques 
of constructing spaces of continuous functions with few operators 
developed by the first named author and other authors,
in particular by I. Schlackow (\cite{iryna}) and the
 techniques of S. Shelah developed for
constructing endo-rigid Boolean algebras in \cite{endorigid86} and \cite{endorigid11}.
Both of these methods are related to rigidity of a compact $K$. For 
a compact $K$ introduce the following notions:
\begin{enumerate}[(a)]
\item $K$ is piecewise strongly rigid, if for every continuous $\phi: K\rightarrow K$
there is a partition $U_1\cup ...\cup U_k=K$ of $K$ into pairwise disjoint
clopen sets $U_1, ..., U_k$  for some $k\in \N$ such that $\phi\restriction U_i$ is either constant or the identity,
\item $K$ is  strongly rigid, if  every continuous $\phi: K\rightarrow K$
 is either constant or the identity,
\item $K$ satisfies the weak$^*$ rigidity condition, if for every $\phi: K\rightarrow M(K)$ where
$M(K)$ is space of Radon measures on $K$ with the weak$^*$ topology (induced
from $C(K)$)  the set $\{\tau(x)|(K\setminus\{x\})\mid x\in K\}$
is relatively weakly compact in the weak topology on $M(K)$.
\end{enumerate}
Assuming that $K$ has no butterfly points
condition (a) for the Stone space $K_{\mathcal A}$ of a Boolean algebra $\mathcal A$
is equivalent for  the algebra $\mathcal A$ to be endo-rigid. 
Condition (c)  is equivalent for $C(K)$ to have few$^*$ operators  (Theorem 23 of \cite{fewsur}).
For $K$ connected (c) implies (b) and (b) is equivalent to (a) (cf. \cite{centripetal}).
However (c) and (a) are not equivalent either in connected or totally disconnected situation.
 The classical space satisfying (b) and not (c)
is the Cook continuum (\cite{cook}) and arbitrarily big spaces constructed by Trnkova (\cite{trnkova}).
The former space is a metrizable continuum, so by the Milutin theorem
the corresponding $C(K)$ has as many operators as $C([0,1])$ and that is why it fails (c).
A totally disconnected  space satisfying (a) and not (c) is the Stone space
of a Boolean algebra $\mathcal A$
minimally generated in the sense of Koppelberg (\cite{minimal})
and  endo-rigid. As proved by Borodulin-Nadzieja in \cite{pbn},
the  Banach spaces $C(K_{\mathcal A})$ is not Grothendieck,
but $C(K)$s which have few$^*$ must be Grothendieck (2.4 \cite{few}).

Both types of constructions of endo-rigid Boolean algebras and 
rigid Banach spaces $C(K)$  can be traced back to the
 the papers \cite{monk} of Monk and  \cite{haydon} of Haydon   which
surprisingly present  practically the same constructions focusing on these different topics. 

So our construction needs a stronger property than the constructions from \cite{endorigid86} and \cite{endorigid11}.
The usual constructions of $C(K)$ spaces with few
or few$^*$ operators (\cite{few}, \cite{plebanek}, \cite{iryna},\cite{antonio})
consisted of obtaining the above weak$^*$ topological rigidity (c), and hence
few$^*$ operators, by constructing $K$ with asymmetric distribution of
separations, for example, in the sense that  given a sequence $\{x_n: n\in \N\}\subseteq E$
for some dense $E\subseteq K$
and a sequence  $(U_n)_{n\in \N}$ of open subsets of $K$ such that $x_n\not\in U_n$
we have $\overline{\{x_n: n\in M\}}\cap \overline{\{x_n: n\in \N \setminus M\}}
\not=\emptyset$ while $\overline{\bigcup {\{U_n: n\in M\}}}\cap \overline{\bigcup {\{U_n: n\in \N \setminus M\}}}
=\emptyset$ for some infinite and coinfinite $M\subseteq \N$ (see \cite{fewsur} Theorems 24 and 25). 
It is clear that this method puts an upper bound of the density of $C(K)$
which is related to the number of all separable compact nonhomeomorphic Hausdorff spaces.
In this paper we formulate 
a new asymmetry condition depending on additional 
parameters which incorporates the ideas of \cite{endorigid86} and \cite{endorigid11}
in the context of weak$^*$ rigidity and connected spaces:

\begin{definition} \label{paracomplicated}
Let $\kappa$ be a cardinal,  $K$ be 
a compact Hausdorff space with a open basis $\B$ and 
let $d_\alpha: K\rightarrow [0,1]$ be continuous for every $\alpha<\kappa$.
Let $d_{\alpha, 1}=d_\alpha$ and $d_{\alpha, -1}=1-d_\alpha$.
$C(K)$ is said
to have asymmetric distribution of separations in the direction of
$\mathcal D=(d_\alpha:\alpha<\kappa)$ if and only if 
  \par
  \noindent {\underline{Given}}

  \begin{enumerate}[(i)]
	\item  $(f_n)_{n\in \N}\subseteq C(K)$ such that $f_n\cdot f_m=0$
for all distinct $n, m\in \N$, $f_n: K\rightarrow [0,1]$ continuous  and 
\item a pairwise disjoint $(U_n)_{n\in \N}\subseteq\B$
	  such that 
$supp(f_n)\cap {U_n}=\emptyset\ \  \hbox{for all $n,m\in\N$},$
	\item $(\nu_n^\xi)_{n\in\N} \subseteq \{-1, 1\}$ for all $\xi\in\kappa$,
	\item $\{\, (U_n^\xi)_{n \in \N} \mid \xi \in \kappa \,\}\subseteq \B$ 
	  such that $U_{n}^{\xi} \subseteq U_n$ for every $n\in\N$ and $\xi \in \kappa$;
  \end{enumerate}

  \noindent \underline{There exist}  an increasing sequence $(\eta_n)_{n\in\N}\subseteq \kappa$ 
  and an infinite, coinfinite $M \subseteq \N$ such that
	  \begin{enumerate}[(a)]
		\item the supremum $ \bigvee_{n\in M} \big(\, f_n\cdot d_{\eta_n, \nu_n^{\eta_n}}\, \big)$ exists in $C(K)$,
		\item  $ \overline{\bigcup_{n\in M}U_{n}^{\eta_n}}\cap
\overline{\bigcup_{n\in \N\setminus M}U_{n}^{\eta_n}}\not=\emptyset$.
	  \end{enumerate}
\end{definition}

\noindent  Section 2 of the paper is devoted to proving the following theorem
(recall that a topological space has c.c.c. if it does not contain an uncountable
family of pairwise disjoint nonempty open sets):
\vskip 6pt
\begin{reptheorem}[\ref{theorem2}] Suppose that  $K$ and $\mathcal D$
are as in Definition \ref{paracomplicated}. 
If $C(K)$ has asymmetric distribution of separations in the direction of
$\mathcal D$ and $K$ is c.c.c., then $C(K)$ has few$^*$ operators.
\end{reptheorem}
\vskip 6pt
This is done by modifying 
proofs of previously considered asymmetric conditions and using some stronger
extraction principles (\ref{rosenthal}, \ref{talagrand}) also proved in this section.  
Section 3 is devoted to the reformulation of the existing theory
concerning the inverse limit constructions ensuring  asymmetric distribution
of separations in previously considered senses.  
While the conditions from Definition \ref{paracomplicated} can be rephrased
in the Banach algebra language (although the final result, Theorem \ref{theorem1}
concerns only the Banach space structure), the proof techniques concerning
separations and the connectedness involves the topological arguments in
$K$. So the main object in Section 3 is the concrete representation  $\nabla\F$
of the Gelfand space of the Banach algebra $[\F]$ generated by subsets $\mathcal F\subseteq C(L)$
for some extremally disconnected $L$.

In section 4 we introduce a concrete type of an inverse limit of compact spaces construction
which on the level of the space of continuous functions 
is called a ladder family (Definition \ref{ladder}). The lemmas from Section 3 are used there
to prove that if $\F\subseteq C(K)$ is a ladder family, then
$C(\nabla\F)$ is connected, has no butterfly points and  provides 
a fertile environment for both the existence of suprema and
nonseparated pairwise disjoint sequences of open sets in $\nabla\F$ needed
to obtain the properties from Definition \ref{paracomplicated}.

In Section 5 we use the combinatorial principle $\diamondsuit(E_\omega^\kappa)$
which follows, by a result of Gregory, from the generalized continuum hypothesis for any regular
uncountable cardinal $\kappa$ to perform a particular construction of
a ladder family.  The character of $\diamondsuit(E_\omega^\kappa)$ as a 
prediction tool allows us to balance the amount of
the suprema and nonseparated pairwise disjoint sequences of open sets to
obtain the conditions from Definition \ref{paracomplicated}. The main theorem
of Section 5 completing the list of all ingredients needed to obtain Theorem \ref{theorem1}
is:
\vskip 6pt
\begin{reptheorem}[\ref{theorem3}] Assume the generalized continuum hypothesis. Let $\kappa$
be the successor of a cardinal of uncountable cofinality.
There is a compact Hausdorff connected c.c.c. space $K$ of weight
$\kappa$ without a butterfly point  such that $C(K)$ has asymmetric distribution of separations
in the direction of some $\mathcal D\subseteq C_I(K)$.
\end{reptheorem}
\vskip 6pt
We do not know if the hypothesis of the generalized continuum hypothesis can be removed from Theorem
\ref{theorem1}. In \cite{endorigid11} Shelah's black boxes were used  to avoid any additional
set theoretic assumption in the construction of endo-rigid Boolean algebras. The Banach space
construction seems more demanding in this context. The first and the third named authors would like to
thank Gabriel Salazar for discussions concerning Shelah's black boxes.

The obtained spaces $C(K)$ have other usual properties of $C(K)$s 
with few operators proved in \cite{few} such as having proper subspaces,
in particular hyperplanes not isomorphic to
the entire space, not being isomorphic to $C(L)$ for $L$ totally disconnected etc.
One could point out one property not mentioned in the literature that
the space $C_\C(K)$ of complex valued functions of $K$ is an indecomposable
complex Banach space which additionally carries the structure of a 
commutative $C^*$-algebra\footnote{
To see this look at $C_\C(K)$ as $C(K)\oplus C(K)$ with the
multiplication by a complex scalar defined as $(\alpha+i\beta)(f, g)=(\alpha f-\beta g, \beta f+\alpha g)$.
A linear operator on $C_\C(K)$ can be identified with a $2\times 2$ matrix $A$
of operators on $C(K)$ such that $T(f, g)=A(f, g)$. The $\C$-linearity of $T$ 
imposes the condition $iT(1, 0)=T(0,1)$ which yields (by $i(f, g)=(-g, f)$):
\[T(f, g)=\begin{bmatrix}
    T_1 & -T_2  \\
    T_2 & T_1 
  \end{bmatrix}\begin{bmatrix}
    f   \\
    g 
  \end{bmatrix}\]
for some operators $T_1, T_2$ on $C(K)$. If $C(K)$ has few operators, this reduces
to a sum of a matrix of weakly compact operators and 
an operator of multiplication by a complex function $M_h(f+ig)=(h_1+ih_2)(f+ig)=
(h_1f-h_2g)+i(h_2f+h_1g)$ for some $h_1, h_2\in C(K)$. Hence as
in the real case every projection $P$ on $C_\C(K)$ is of the form $hI+S$ for $h\in C_\C(K)$
and $S$ strictly singular, and the condition $P^2=P$ yields that $h^2=h$ as
no multiplication can be strictly singular for continuous functions on
a $K$ with no isolated points ($K$ is connected). Hence $h(x)=0$ or $h(x)=1$
for each $x\in K$ and so $h=1$ or $h=0$ since $K$ is connected.
It follows that $P=I+S$ of $P=S$ where $S$ is finite dimensional since $S$ is a projection
as well, and so $C_\C(K)$ is indecomposable indeed.}.

Making considerably less technical effort and following
the ideas of this paper one could construct a totally disconnected
$K$ of arbitrarily big size such that $C(K)$ has few operators. This would already
provide Banach spaces of densities $\kappa$, for arbitrarily big $\kappa$
 without complemented infinite dimensional subspaces of
densities less than $\kappa$ giving a strong negative answer to question (B). We opted 
for presenting just the connected example. Our $K$ has one additional peculiar
property, while it has no nontrivial convergent sequence (this would give rise to a complemented
copy of $c_0$) for any pairwise disjoint sequence $(U_n)_{n\in \N}$ of nonempty open subsets of $K$
there are only countably many sets $M\subseteq \N$ such that $\overline{\{U_n\mid n\in M\}}
\cap \overline{\{U_n\mid n\in \N\setminus M\}}=\emptyset$. This follows from Lemma 
\ref{separating-countable} and the construction.

The terminology and notation of the paper should be standard. In set theory
we follow \cite{jech}, \cite{kunen}, in topology \cite{engelking}, in Boolean algebras
\cite{koppelberg}, \cite{halmos}
in Banach spaces \cite{fabianetal}, \cite{semadeni}. Important conventions include:
\begin{itemize}
\item $C_I(K)=\{f\in C(K)\mid f:K\rightarrow [0,1]\}$,
\item $(f_n)_{n\in \N}$ are pairwise disjoint if $f_n\cdot f_m$ for all distinct $n,m\in \N$.
\item GCH is the generalized continuum hypothesis i.e., the statement that
$2^\kappa$ is the successor cardinal $\kappa^+$ for every infinite cardinal $\kappa$.
\item  
$E^\kappa_\omega = \{\alpha \in \kappa\mid\text{cf}(\alpha)=\omega\}$ denotes the set of ordinals smaller than $\kappa$ of cofinality $\omega$. 
\item  $\supp (f)$ denotes $f^{-1}[\R\setminus \{0\}]$  for
any real valued function $f$.
\end{itemize}

\section{Few$^*$ operators from asymmetric distributions of separations}

The purpose of this section is to prove Theorem \ref{theorem2}.
This amounts to applying the theory initiated in \cite{few} and later
developed in \cite{big}, \cite{grande}, \cite{plebanek}, \cite{iryna}, \cite{fewsur} in the
new context of the sequence $\D=\{d_\alpha:\alpha<\kappa\}$.  The first step is to
prove  a Rosenthal  lemma type extraction principle
 in the flavour of the approach from Chapter 4.3 of \cite{iryna} (cf. Theorems 24 and 25 of \cite{fewsur}):

\begin{lemma} \label{rosenthal}
  Let $K$ be a compact Hausdorff space,  $T:C(K)\rightarrow C(K)$ 
a bounded linear operator and let $\varepsilon>0$.
  Let $(f_n)_{n\in\N}\subseteq C_I(K)$ be pairwise disjoint and 
let  $(U_n)_{n\in\N}$ be a pairwise disjoint family of nonempty open subsets  of $K$.
  Then there are an infinite $M\subseteq\N$ and 
nonempty open sets $V_n\subseteq U_n$ for $n\in\N$
  such that for all $m\in M$ and for all sequences 
$(g_n)_{n\in\N}\subseteq C_I(K)$ satisfying $g_n\leq f_n$ for $n\in\N$ we have
  \[
	\sup_{x\in V_m} \sum_{n\in M\setminus\{m\}} |T(g_n)(x)| < \varepsilon.
  \]
\end{lemma}
\begin{proof}
  Let us introduce an auxiliary notation: for sets $M\subseteq N \subseteq \N$ and two 
pairwise disjoint sequences  $(V_n)_{n\in M}$, 
  $(U_n)_{n\in N}$ of nonempty
open subsets of $K$ we 
write $(V_n)_{n\in M} \prec (U_n)_{n\in N}$ if $V_n\subseteq U_n$ for all $n\in M$.

  By recursion on $k\in \N$ construct
  \begin{itemize}
    \item $n_1< ...<n_k$ in $\N$,
    \item $(V_{n_1}, ..., V_{n_k})\prec (U^0_{n_1}, ..., U^{k-1}_{n_k}) \prec (U_{n_1}, ..., U_{n_k})$,
    \item $\N=X_0\supseteq X_1\supseteq ... \supseteq X_k$ such that
 $X_k$ is an infinite subset of $\N\setminus[1,\,n_k]$,
    \item $(U_n)_{n\in \N}=(U_n^0)_{n\in X_0}\succ (U_n^1)_{n\in X_1}\succ\dots\succ (U_n^k)_{n\in X_k}$.
  \end{itemize}
  such that
  \[
    |T(g_{n_k})(x)|\leq {\varepsilon \over{2^{k+1}}}\leqno (\ast)
  \]
  holds  for all $0\leq g_{n_k}\leq f_{n_k}$ and for all  $x\in U_n^k$ with $n\in X_k$.
  Moreover
  \[
    \sum_{n\in X_k}|T(g_n)(x)|\leq {\varepsilon \over{2^k}}\leqno (\ast\ast)
  \]
  holds for all $x\in V_{n_k}$ and for all $0<g_n\leq f_n$ with $n\in X_k$.
  
  As $n_0$ is undefined, the above is vacuously true for $k=0$.  So, suppose we have
  the above objects for $k\geq0$ and let us  construct the corresponding objects for
  $k+1$.
  Note that $(\ast)$ and $(\ast\ast)$ are worded in such a way that given $X_k$ we need to find
  $n_{k+1}\in X_k$ and an infinite $X_{k+1}\subseteq X_k\setminus[1,\,n_{k+1}]$
  such that  $(\ast)$ and $(\ast\ast)$ are satisfied for $k+1$ in place of $k$. That is,
  the previous $(n_1, ..., n_k)$ and $(V_{n_1}, ..., V_{n_k})$ play no role  when we
  pass to $(\ast)$ and $(\ast\ast)$ for $k+1$. First we will take care of ($\ast$).
  
  Suppose  that there is no $n_{k+1},\ X_{k+1}$ and $U^{k+1}_n$ for $n\in X_{k+1}$
such that $(\ast)$ holds, 
  that is, for all $n'\in X_k$, all infinite 
  $X\subseteq X_k\setminus [1,\,n']$, and all $(U'_n)_{n\in X}\prec (U_n^k)_{n\in X_k}$
  there exist a number $n''\in X$, and $0\leq g_{n',n''}\leq f_{n'}$, and an element $x_{n''}\in U'_{n''}$ such that
  \[
    |T(g_{n',n''})(x_{n''})| > {\varepsilon \over{2^{k+2}}}.\leqno (\ast\ast\ast)
  \]
  We will derive   contradiction from this hypothesis.
   Let $l\in \N$ be such that ${{l\varepsilon}\over{2^{k+2}}}>\|T\|$.
   Applying the above recursively on $i\leq 2l$ we can construct (note that the index $k+1$ below 
is fixed and indicates only that we are in the ($k+1$)-th stage of the recursive construction):
  \begin{itemize}
    \item an increasing $(n_{k+1}^i)_{i\leq 2l} \subseteq X_k$ with $n^1_{k+1}>n_k$,
    \item  an infinite $X^i_{k+1}$ such that $X^{i+1}_{k+1}\subseteq X^{i}_{k+1}\subseteq X_{k}\setminus [1,\, n_{k+1}^i]$,
  \item nonempty open $W_n^i$s for $n\in X^i_{k+1}$ such that 
  \[
    (W_n^{i+1})_{n\in X^{i+1}_{k+1}}\prec(W_n^{i})_{n\in X^{i}_{k+1}}\prec (U_n^k)_{n\in X_k}
  \]
  \item  $0\leq g_{n^i_{k+1},n} \leq f_{n^i_{k+1}}$ for $n\in X^i_{k+1}$ such that 
  \[
    |T(g_{n^i_{k+1},n})(x)|> {\varepsilon \over{2^{k+2}}}
  \]
  for all $x\in W_n^i$, and all $n\in X^i_{k+1}$.  
  \end{itemize}
  
  To move from $i$ to $i+1$ we set $n^{i+1}_{k+1}=\min X^i_{k+1}$ and use repeatedly the above hypothesis for each $j\in\N$
  with $n'=n^{i+1}_{k+1}$, $X= X^i_{k+1} \cap [n''_{j-1},\infty)$, and $(U'_n)_{n\in X} = (W^i_n)_{n\in X}$ to obtain  
    $n''_j\in X,\ g_{n',n''_j}\leq f_{n'}$, and $x_{n''_j} \in U'_{n''_j}$. Then $X^{i+1}_{k+1}=(n''_j)_{j\in\N}$ and
    we use the continuity of $|T(g_{n',n''_j})|$ to conclude that
   if it is bigger than ${\varepsilon \over{2^{k+2}}}$ at point $x_{n''_j}$, then it is bigger than
   ${\varepsilon \over{2^{k+2}}}$ at some neighborhood $W^i_{n''_j}$ of that point.
  
   Arriving at $i=2l$ we set $m=\min X^{2l}_{k+1}$, pick $x_0 \in W^{2l}_m$ and fix a finite set 
   $F\subseteq [1,\,2l]$ of cardinality not less than $l$ such that all numbers 
   $T(g_{n^i_{k+1}, m})(x_0)$ have the same sign for $i\in F$. 
   Then we have
   \[
     |T(\sum_{i\in F}g_{n^i_{k+1}, m})(x_0)|> {l\varepsilon \over{2^{k+2}}}\geq ||T||.
   \]
   This is a contradiction since the norm of  $\sum_{i\in F}g_{n^i_{k+1},m}$ is less than or equal to one.
  
  Hence our hypothesis was false, that is, there is $n_{k+1}\in X_k$ and an infinite 
  $X_{k+1}'\subseteq X_k\setminus[1,\, n_{k+1}]$
  such that for some nonempty $U_n^{k+1}\subseteq U_n^k$
 with $n\in X_{k+1}'$ the condition $(\ast\ast\ast)$ holds.
  That is $(\ast)$ holds for $k+1$ in place of $k$.
  
  Now we will choose  a nonempty $V_{n_{k+1}}\subseteq U_{n_{k+1}}^k$  and
  an infinite $X_{k+1}\subseteq X_{k+1}'$ such that $(\ast\ast)$ holds for $k+1$ instead of $k$. 
  Let
  \[
    s=\sup_{x\in U_{n_{k+1}}^k} \{\sum_{n\in X_{k+1}'}|T(g_n)(x)| \mid 0\leq g_n\leq f_n\}.
  \]
  Note that $s\leq 2\|T\|$ as for the supremum we can consider finite sums of numbers with constant sign, 
  which by the linearity of $T$ are reduced to values of the operator $T$ on vectors of norm less than or equal to one. 
  Choose $x_0\in  U_{n_{k+1}}^k$ such that
  \[
    s-\sup \{\sum_{n\in X_{k+1}'}|T(g_n)(x_0)| \mid 0\leq g_n\leq f_n)\}<{{\varepsilon}\over{2^{k+3}}},
  \]
  and then a finite $F\subseteq X_{k+1}'$ and $0\leq g_n\leq f_n$ for $n\in F$ such that
  \[
    s-\sum_{n\in F}|T(g_n)(x_0)|<{{\varepsilon}\over{2^{k+2}}}.
  \]
  Now, note that by the continuity of the functions $T(g_n)$ for $n\in F$ at $x_0$ there
  is a nonempty neighborhood of $x_0$ of the form $V_{n_{k+1}}$ for
  $V_{n_{k+1}}\subseteq U_{n_{k+1}}^k$ where the above inequality holds.  Put $X_{k+1}=X_{k+1}'\setminus F$
  and note that by the choice of $s$ and $F$  we have $(\ast\ast)$ with $k+1$ in the place of $k$.
  
  This completes the recursive construction.  Note that $n_{k+1}\in X_k$  for each $k\in \N$.
  To verify the statement of the lemma let $M=(n_k)_{k\in \N}$ and choose a sequence $(g_n)_{n\in M}
\subseteq C_I(K)$ with
  $g_n\leq f_n$ for all $n\in M$ and $m=n_k\in M$ and $x\in V_{n_k}$. Then
  \[
    \sum_{n\in M\setminus\{m\}} |T(g_n)(x)|\leq \sum_{1\leq i<k} |T(g_i)(x)|+ \sum_{n\in X_{k}} |T(g_n)(x)|
  \]
  as $M\setminus\{n_1, \dots, n_k\}\subseteq X_k$.
  The first sum is not bigger  than
  \[
    \sum_{1\leq i<k}{\varepsilon\over 2^{i+1}}={\varepsilon\over 2}(1- 2^{-k+1})\leq \varepsilon/2\] 
   by applying $(\ast)$ since
  $n_k\in X_{i}$ for each $i<k$ and $V_{n_k}\leq U_{n_k}^i$.
  On the other hand the second sum is not bigger than $\varepsilon/2$ by
  applying directly $(\ast\ast)$. Hence we obtain the statement of the lemma.
\end{proof}

Another extraction principle which we will need is the following:

\begin{lemma}\label{talagrand} Suppose that $V_n$'s for $n\in\N$ are 
pairwise disjoint open sets in a compact space $K$ and
$x_n\in K\setminus {V_n}$ are distinct. Suppose that $\varepsilon>0$ and  $\mu_n$ for $n\in \N$ is a 
Radon measure on $K$ such that $|\mu_n|(V_n)>\varepsilon$ for all $n\in \N$. 
Then there are:  an infinite $M\subseteq \N$, open $V_n'\subseteq V_n$ and a $\delta>0$  such that
for all $n\in M$ we have that $|\mu_n|(V_n')>\delta$ and
\[{\overline{\bigcup_{n\in M}V_n'}}\cap\{x_n: n\in M\}=\emptyset.\]
\end{lemma}
\begin{proof} By going to a subset we may assume that $\{x_n: n\in \N\}$
forms a discrete subspace of $K$. By the regularity of the measures and by going to subsets
of $V_n$s we may assume that $(\overline{V_n})_{n\in \N}$
is pairwise disjoint and $x_n\not\in \overline{V_n}$ for each $n\in \N$.
Consider a coloring $c:[\N]^2\rightarrow \{0,1, 2\}$ defined for distinct $n,m\in \N$ by
$c(\{n, m\})=0$ if $x_n\in \overline V_m$ and $n<m$,  $c(\{n, m\})=1$ if 
the previous condition does not hold and $x_m\in \overline V_n$ and $n<m$,  
and $c(\{n, m\})=2$ if $\{x_n, x_m\}\cap(\overline V_n\cup\overline V_m)=\emptyset$. 
Apply the Ramsey theorem for $c$ 
obtaining an infinite subset of $\N$ which is homogenous for $c$.
However,  a three element $0$-homogenous set or $1$-homogenous set would contradict 
the pairwise disjointness of $\overline V_n$s, so we have
an infinite $2$-homogenous set. Hence, by going to a subset we may assume
that $x_n\not\in \overline V_m$ for any two $n, m\in \N$. Let $U_n$ be an
open neighbourhood of $x_n$ such that $U_n\cap V_m=\emptyset$ for all
$m\leq n$ in $\N$.
We will consider two cases. 
\vskip 6pt
\noindent Case 1. There is $\delta>0$ and a point $x\in K$ 
 such that for
each open
neighbourhood $W$ of $x$ the set  $\{n\in \N: |\mu_n|(V_n\cap W)>\delta\}$ is infinite.

\noindent As $V_n$'s  are   pairwise disjoint
and by the regularity of the measures by going from $V_n$ to its 
subset we may assume that $x\not\in\overline V_n$
for every $n\in \N$.  Further removing at
most one index  we may assume that $x\not\in\{x_n: n\in \N\}$

  Now recursively define a decreasing sequence $(W_k)_{k\in \N}$
of open neighbourhoods of $x$  
and a strictly  increasing sequence $(n_k)_{k\in \N}\subseteq \N$  such that
the following two conditions hold:
\[|\mu_{n_k}|\big(V_{n_k}\cap (W_k\setminus \overline W_{k+1})\big)>\delta\leqno (1)\]
\[x_{n_k}\not\in \overline{W_{{k+1}}}.\leqno (2)\]
This is possible by the hypothesis of Case 1.  Put 
$V_{n_k}'=V_{n_k}\cap (W_k\setminus \overline W_{k+1})$. It follows that
\[\overline{\bigcup_{k\in \N}V_{n_k}'}\setminus \bigcup_{k\in \N}{\overline {V_{n_k}'}}\subseteq
 \bigcap_{k\in \N}\overline W_k,\]
which is disjoint from $\{x_{n_k}: k\in \N\}$ by (2). 
Since $\{x_{n_k}: k\in \N\}$ is disjoint form $\bigcup_{n\in \N}{\overline V_{n}}$
by the argument before Case 1, we conclude the Lemma in this case for $M=\{n_k: k\in \N\}$.

\vskip 6pt
\noindent Case 2. Case 1 does not hold.\par

\noindent Since the hypothesis of Case 1 fails, for every $n\in \N$
and for every $\delta'>0$ there is an $m(n,\delta')\in \N$
and an open neighbourhood $W(n,\delta')$ of $x_n$  such
that
\[|\mu_k|(V_k\cap {\overline{W(n,\delta')}})<\delta'\leqno (1)\]
for all $k>m(n,\delta')$.
 Thus, one can choose 
recursively a strictly
increasing  sequence of integers $(k_n)_{n\in N}$
such that $k_n>m(k_j, {\varepsilon\over{2^{j+2}}})$
for all $j<n$.
Consider
\[V_{k_n}'=V_{k_n}-\bigcup
\{{\overline{W(k_j,{\varepsilon\over{2^{j+2}}})}}:j<n\}.\]
By (1) we have $|\mu_{k_n}|(V_{k_n}\cap
 {{\overline{W(k_j,{\varepsilon\over{2^{j+2}}})})}}<
{\varepsilon\over{2^{j+2}}}$
for $j<n$ and so
\[|\mu_{k_n}|(V_{k_n}')>\varepsilon/2=\delta.\leqno (2)\]
\noindent Now, note that $W(k_j, {\varepsilon\over{2^{j+2}}})$
is disjoint from $V_{k_n}'$ for $n>j$ and hence
$W(k_j, {\varepsilon\over{2^{j+2}}})\cap U_{k_j}$ is an
open neighbourhood of $x_{k_j}$ witnessing the fact
that $x_{k_j}\not\in \overline{\bigcup_{n\in \N}V_{k_n}'}$ (Recall that $U_n$s are
open neighbourhoods of $x_n$s such that $U_n\cap V_m=\emptyset$ for all
$m\leq n$ in $\N$). So
this proves the lemma for  $M=\{n_k: k\in \N\}$ and $\delta=\varepsilon/2$.

\end{proof}

Now recall Definition \ref{multiplier}  and the following characterization
of weak multipliers:

\begin{theorem}[\cite{few} 2.1., 2.2. ]\label{weak-multiplier}
  Let $K$ be a compact Hausdorff space and let $T: C(K)\rightarrow C(K)$
be a bounded linear operator. The following conditions are equivalent:
  \begin{enumerate}
	\item $T$ is a weak multiplier,
	\item for every pairwise disjoint sequence $(f_n)_{n\in\N}\subseteq C_I(X)$  and every sequence $(x_n)_{n\in\N}\subseteq K$ such that
	      $f_n(x_n)=0$ for all $n\in\N$ we have
		  \[
			\lim_{n\rightarrow\infty} T(f_n)(x_n) = 0.
		  \]
  \end{enumerate}
\end{theorem}

\begin{lemma} \label{noncentripetal}
  Let $K$ be a compact Hausdorff space and let $T: C(K)\rightarrow C(K)$ be a
bounded linear operator. If $T$ is not a weak multiplier, then
  there exist $\varepsilon>0$, a pairwise disjoint sequence $(f_n)_{n\in\N}\subseteq C_I(K)$ and 
  a pairwise disjoint sequence  $(U_n)_{n\in\N}$  of nonempty open subsets of $K$ such that
  \[
    \supp(f_n) \cap U_m = \emptyset \text{ for all } n,m\in\N
  \]
  and 
  \[
	|T(f_n)|\restriction U_n > \varepsilon \text{ for all } n\in\N.
  \]
\end{lemma}
\begin{proof}
First let us prove the forward implication.  By Theorem \ref{weak-multiplier}(2) there 
is a pairwise disjoint sequence 
$(g_n)_{n\in\N}\subseteq C_I(K)$ and a sequence $(x_n)_{n\in\N}\subseteq K$ such that
	      $g_n(x_n)=0$ for all $n\in\N$,   and 
			$|T(g_n)(x_n)|>\varepsilon'$  for some $\varepsilon'>0$.
Let $V_n=\supp(g_n)$ and $\mu_n=T^*(\delta_{x_n})$.
Note that the hypothesis of Lemma \ref{talagrand} is satisfied
for  $\varepsilon'$ instead of $\varepsilon$, so we may find
appropriate $\delta>0$, and infinite $M\subseteq \N$ and $V_n'\subseteq V_n$ for each $n\in \N$.
Let $U_n'$ be open neighbourhoods of $x_n$ for $n\in M$ such  that $\overline{\bigcup_{n\in M}V_n'}\cap
U_n'=\emptyset$. Since $\{x_n: n\in \N\}$ may be assumed to be discrete
by going to a subsequence,
we may assume that $U_n'$s are pairwise disjoint. Now choose $f_n\in C_I(K)$
such that $\supp(f_n)\subseteq V_n'$ and that $|\int f_nd\mu_n|>\delta/2$, this can be done
since $|\mu_n(V_n')|>\delta$. It follows that $|T(f_n)(x_n)|>\delta/2$ for each $n\in \N$.
Now find $U_n\subseteq U_n'$  such that $|T(f_n)\restriction U_n|>\delta/2$. By 
re-enumerating $M$ and putting $\varepsilon=\delta/2$ we obtain the statement from the lemma.

For the backward implication pick an $x_n\in U_n$ an apply \ref{weak-multiplier}.
\end{proof}

\begin{theorem}\label{theorem2} Suppose that $K$ and $\mathcal D$
are as in Definition \ref{paracomplicated}. 
If $C(K)$ has asymmetric distribution of separations in the direction of
$\mathcal D$ and $K$ is c.c.c.,  then $C(K)$ has few$^*$ operators.
\end{theorem}
\begin{proof} Let $\kappa$ and $\B$ be as in Definiton \ref{paracomplicated}.
If $C(K)$ does not have few$^*$ operators, then by Definition \ref{multiplier}
 there is bounded linear operator 
$T\colon C(K)\rightarrow C(K)$ which is not a weak multiplier.
  Then by \ref{noncentripetal} there are
  \begin{itemize}
    \item a pairwise disjoint sequence $(f_n)_{n\in\N}$ in $C_I(K)$, 
    \item a pairwise disjoint sequence $(U_n)_{n\in\N}$ in $\B$ such that
      \[ \supp(f_n) \cap U_m = \emptyset, \text{ for all } n,m \in \N,\]
    \item $\varepsilon>0$
  \end{itemize}
  and for each $n\in \N$ we have \[|T(f_n)|\restriction U_n > 2\varepsilon.\]
  
  Now by applying \ref{rosenthal} for $\varepsilon/3$ we may assume
  that for any $m\in \N$ and for any sequence $(g_n)_{n\in\N} 
\subseteq C_I(K)$ such that $g_n\leq f_n$ for all $n\in \N$
    \[
  	\sup_{x\in U_m} \sum_{n\in M\setminus\{m\}} |T(g_n)(x)| < \varepsilon/3. \leqno(\ast)
    \]
  
  To make use of the asymmetric distribution of separations in the direction
of $\mathcal D$  we need to construct the following:
  \begin{itemize}
    \item $\{(\nu_{n}^{\xi})_{n\in \N} \mid \xi \in \kappa\} \subseteq \{\pm1\}$,
    \item $\{\, (U_{n}^{\xi})_{n \in \N} \mid \xi \in \kappa \,\}\subseteq \B$ 
  	  satisfying  $U_{n}^{\xi} \subseteq U_n $ for every $n\in\N$ and every $\xi \in \kappa$.
  \end{itemize}
We will construct the above objects in such a way 
that
  for all $n\in \N$ and all $\xi \in \kappa$ we have
 $|T(f_nd_{\nu_{n}^{\xi},\xi})|\restriction U_n^{\xi}>\varepsilon.$
  This is achieved in the following way. 
  Fix $\xi\in \kappa$, $n\in\N$ and $x_n \in U_n$.  Since $f_n=f_nd_{1,\xi}+f_nd_{-1,\xi}$,
  we have either
  \[
    |T(f_nd_{1,\xi})(x_n)| >\varepsilon \text { or }
    |T(f_nd_{-1,\xi})(x_n)|>\varepsilon.
  \]
  We choose $\nu_{n}^{\xi} \in \{\pm1\}$, the one for which the above holds 
  and define $U_{n}^{\xi}\in \B$ to be an open neighborhood of $x_n$ included in $U_n$ such that  
  \[
    |T(f_nd_{\nu_{n}^{\xi},\xi})|\restriction U_{n}^{\xi} >\varepsilon.\leqno(**)
  \]
  This completes the construction of $\{(\nu_{n}^{\xi})_{n\in\N} \mid \xi\in \kappa\}$
  and $\{(U_{n}^{\xi})_{n\in \N} \mid  \xi \in \kappa\}$.
  
  Let us fix an almost disjoint family $\{N_\alpha \mid \alpha<\omega_1\}$ of infinite subsets of $\N$.
  We will be considering the sets:
  \begin{itemize}
    \item $F_\alpha=(f_n)_{n\in N_\alpha}$,
    \item $U_\alpha=(U_n)_{n\in N_\alpha}$.
    \item $\mathcal N_\alpha=\{(\nu_{n}^{\xi})_{n\in N_\alpha} \mid \xi \in \kappa\}$,
    \item $\U_\alpha=\{(U_{n}^{\xi})_{n\in N_\alpha} \mid \xi \in \kappa\}$.
  \end{itemize}
  
  We use the hypothesis  that $C(K)$ has asymmetric distribution of separations in the direction of $\D$
 for $F_\alpha$, $U_\alpha$, 
  $\mathcal N_\alpha$, and $\U_\alpha$ for each $\alpha<\omega_1$ obtaining
increasing sequences 
  $(\eta^\alpha_n)_{n\in\N}\subseteq \kappa$ and infinite sets $M_\alpha\subseteq N_\alpha$ such that 
  \begin{enumerate}
    \item $\bigvee_{n\in M_\alpha}f_nd_{\nu_{n}^{\eta^\alpha_n},\eta^\alpha_n}$ exists in $C(K)$,
    \item $\overline{\bigcup\{U_{n}^{\eta^\alpha_n}\mid n\in M_\alpha\}}\cap 
\overline{\bigcup\{U_{n}^{\eta^\alpha_n}\mid n\in N_\alpha\setminus M_\alpha\}}\not=\emptyset$.
  \end{enumerate} 
  
  Let us define $g^\alpha_n=f_nd_{\nu_{n}^{\eta^\alpha_n},\eta^\alpha_n}$ 
for all $n\in\N$, $\alpha<\omega_1$.
  If for some $\alpha<\omega_1$ we had
  \begin{itemize}
    \item for all $n\in M_\alpha$, $x\in U_{n}^{\eta^\alpha_n}$: 
      $|T(\bigvee_{n\in M_\alpha} g^\alpha_n)(x)|\geq 2\varepsilon/3$ and
    \item for all $n\in N_\alpha\setminus M_\alpha$, $x\in U_{n}^{\eta^\alpha_n}$: 
      $|T(\bigvee_{n\in M_\alpha} g^\alpha_n)(x)|\leq \varepsilon/3$ 
  \end{itemize}
  then, we would separate the sets in $(2)$  contradicting the condition $(2)$.
  Therefore the conjunction of the above statements are false for each $\alpha<\omega_1$.
  By going to an uncountable subset of $\omega_1$ we may assume that there 
is $\delta>0$ and $n_0\in \N$ such that for each
  $\alpha<\omega_1$ we have some 
$\emptyset \neq V_\alpha \subseteq U_{n_0}^{\eta^\alpha_{n_0}}\subseteq U_{n_0}$ such that
either
  \begin{enumerate}
    \item[(3)] $n_0\in M_\alpha$ and for all $x\in V_\alpha$ we have 
      $|T(\bigvee_{n\in M_\alpha} g^\alpha_n)(x)| < 2\varepsilon/3 - \delta$ or
    \item[(4)] $n_0\in N_\alpha\setminus M_\alpha$ and for all $x\in V_\alpha$ we have 
      $|T(\bigvee_{n\in M_\alpha} g^\alpha_n)(x)| > \varepsilon/3 + \delta$.
  \end{enumerate}
This gives for each $\alpha<\omega_1$ the following statement: \begin{itemize}
    \item $n_0\in \N\setminus M_\alpha'$ and for all $x\in V_\alpha$ we have 
      $|T(\bigvee_{n\in M_\alpha'} g^\alpha_n)(x)| > \varepsilon/3 + \delta$.
  \end{itemize}
Where $M_\alpha'=M_\alpha$ if (4) holds and $M_\alpha'=M_\alpha\setminus\{n_0\}$
if (3) holds. In the latter case the above condition follows from (3) and 
  $(**)$ since then we have $|T(g^\alpha_{n_0})(x)|>\varepsilon$.
 Now consider $m\in \N$ such that $m\delta/2>\|T\|$. 
  Using the c.c.c. of the space $K$ and Lemma
\ref{ccc} we may find $\alpha_1< ... < \alpha_m<\omega_1$ such that
  \[
    V_{\alpha_1} \cap V_{\alpha_2} \cap \dots \cap V_{\alpha_m} \neq \emptyset.
  \]
  So let $x_0$ be a point in $V_{\alpha_1} \cap V_{\alpha_2} \cap \dots \cap V_{\alpha_m} \neq \emptyset$.
  
  Now  take  $k\in \N$ such that for each $1\leq i<j\leq m$ we have
  \[
    (M_{\alpha_i}'\setminus k) \cap (M_{\alpha_j}'\setminus k)=\emptyset.
  \]
  By $(\ast)$ for each $1\leq i\leq m$ we have that
  \[
    \sum_{n\in M_{\alpha_i}'\cap k}|T(g^{\alpha_i}_n)(x_0)|<\varepsilon/3
  \]
  and so for each $1\leq i\leq m$ we have
  \[
    |T\Big(\bigvee_{n\in M_{\alpha_i}'\setminus k}g^{\alpha_i}_n\Big)(x_0)|>\delta.
  \]
  For $m/2$ indices $1\leq i\leq m$, say from a set $F\subseteq \{1, ..., m\}$ all the reals 
  \[
    T\Big(\bigvee_{n\in M_{\alpha_i}'\setminus k}g^{\alpha_i}_n\Big)(x_0)
  \]
  have the same sign, and so
  \[\|T\Big(\sum_{i\in F}\bigvee_{n\in M_{\alpha_i}'\setminus k}g^{\alpha_i}_n\Big)\|\geq
    |T\Big(\sum_{i\in F}\bigvee_{n\in M_{\alpha_i}'\setminus k}g^{\alpha_i}_n\Big)(x_0)| =
    \sum_{i\in F}|T\Big(\bigvee_{n\in M_{\alpha_i}'\setminus k}g^{\alpha_i}_n\Big)(x_0)| > m\delta/2.
  \]
  However  
  \[
    \|\sum_{i\in F}\bigvee_{n\in M_{\alpha_i}'\setminus k}g^{\alpha_i}_n\| \leq 1
  \]
  as $(M_{\alpha_i}'\setminus k)$s are pairwise all disjoint. 
  Using the fact that $m\delta/2>\|T\|$ we obtain a contradiction with the definition of the norm of $T$.
  This means that assuming that there is an operator $T$ on $C(K)$ which is not a weak multiplier 
  leads to a contradiction completing the proof of the proposition.
\end{proof}

\section{Controlling separations and the connectedness}

\subsection{Boolean algebras,  their Stone spaces and continuous functions on the Stone spaces}

Recall that a topological space is called c.c.c. if it does not contain an 
uncountable collection of pairwise disjoint open sets. We have the following:

\begin{lemma}\label{ccc}
Suppose $K$ is a compact Hausdorff space and that $\kappa$ is a cardinal.
\begin{itemize}
\item $\{0,1\}^{\kappa\times\N}$ is c.c.c.
\item If $K$ is c.c.c, $m\in \N$  and $\{V_\xi: \xi<\omega_1\}$ is a collection
of nonempty open subsets of $K$, then there are distinct $\xi_1, ..., \xi_m$
such that $V_{\xi_1}\cap ...\cap V_{\xi_m}\not=\emptyset$.
\end{itemize}
\end{lemma}
\begin{proof} The first condition follows from the Hewitt-Marczewski-Pondiczery
Theorem (2.3.17. \cite{engelking}). For the second condition prove it by induction on $m\in \N$.
For $m=2$ it is the c.c.c.  Given it for $m$, build recursively pairwise disjoint family
of  sets $F_\alpha\subseteq \omega_1$ of cardinality $m$
for $\alpha<\omega_1$ such that $W_\alpha=\bigcap_{\xi\in F_\alpha}V_\xi\not=\emptyset$,
now apply   the c.c.c. for $(W_\alpha)_{\alpha<\omega_1}$.
\end{proof}

In this section we use the following notation:  
\begin{itemize}
\item $\kappa$ will denote an uncountable regular cardinal,
\item If $\mathcal A$ is a Boolean algebra,  $S(\mathcal A)$ denotes the Stone space of $\mathcal A$,
i.e., a compact Hausdorff totally disconnected space such that there is a Boolean isomorphism
between $\mathcal A$ and the algebra of clopen subsets of $S(\mathcal A)$,
\item  the clopen set of $S(\mathcal A)$ corresponding to an element $a$ of $\mathcal A$
will be denoted by $s_{\mathcal A}(a)$,
\item $Fr(\kappa)$ denotes the free Boolean algebra generated by $(e_{\alpha, n})_{\alpha < \kappa, n\in \N}$,
\item For $A\subseteq \kappa$, $Fr(A)$ denotes the subalgebra of $Fr(\kappa)$ generated by
$(e_{\alpha, n})_{\alpha\in A, n\in \N}$,
\item For $A\subseteq \kappa$, $\overline{Fr}(A)$ denotes the Boolean completion of $Fr(A)$,
\item We will identify $\overline{Fr}(A)$ with a subalgebra of $\overline{Fr}(B)$ when
 $A\subseteq B\subseteq \kappa$,
\item For $A\subseteq \kappa$, $L_A$ denotes the Stone space of $\overline{Fr}(A)$,
\item $L_\kappa$ will be denoted by $L$,
\item $I=[0,1]$, 
\item The supremum of a family $\mathcal F$ of functions will be denoted by $\bigvee\F$.
In principle the supremum of the same family of functions can depend on the ambient
lattice of functions, so we will need to add where the supremum is taken.
\end{itemize}

For the Stone duality
or other dual terminology concerning
Boolean algebras  see \cite{koppelberg} or \cite{halmos}, for Gleason spaces see \cite{gleason}. 
The following proposition is the summary of standard facts concerning the
above objects:

\begin{proposition}\label{gleason} $ $
\begin{itemize}
\item For any $A\subseteq \kappa$ the space $L_A$ is the Gleason space of $I^A$,
\item For any $A\subseteq \kappa$ the space $L_A$ is extremally disconnected and c.c.c.,
which implies that bounded subsets of $C(L_A)$ have suprema in the lattice $C(L_A)$,
\item For $A\subseteq \kappa$ there is a continuous surjective map $p_A: L\rightarrow L_A$ dual to
the inclusion $\overline{Fr}(A)\subseteq \overline{Fr}(\kappa)$, in particular
$p_A^{-1}[s_{\overline{Fr}(A)}(a)]=s_{\overline{Fr}(\kappa)}(a)$ for $a\in \overline{Fr}(A)$,
\item For $A\subseteq \kappa$ there is an isometric inclusion $C(L_A)\subseteq C(L)$ induced by $p_A$.
\end{itemize}
\end{proposition}

\begin{lemma}[\cite{grande} Corollary 2.5]  \label{extending-ultrafilters}
  Suppose that  $A, B$ are disjoint subsets of $\kappa$ and 
  $t_A \in L_{A}$ and $t_B\in L_{B}$.
  There exists a point $t\in L$ such that  $p_{A}(t)=t_A$ and $p_B(t)=t_B$.
\end{lemma}

\begin{definition}
For all $\alpha\in\kappa$ we define a Cantor-like surjection $d_\alpha\in C_I(L)$ by
\[d_\alpha(x) = \sum_{n\in\N} {\chi_{s_{\overline{Fr}(\kappa)}( e_{\alpha, n})}(x)\over{2^n}} \in I,\] 
We will use $d_{1,\alpha}=d_\alpha$ and $d_{-1,\alpha}=1-d_\alpha$.
We will also use the notation $\D=\{d_\alpha: \alpha\in \kappa\}$.
\end{definition}

\begin{definition}
We say that $f\in C(L)$ depends on a set $A\subseteq\kappa$,  if $p_A(s)=p_A(t)$
implies $f(s)=f(t)$ for any $s, t\in L$. 
We say that $\F \subseteq C(L)$ depends on a set $A\subseteq \kappa$, if every $f\in\F$ depends on $A$.
\end{definition}

\begin{lemma}\label{d_alpha}
  Let $\alpha\in\kappa$, then
\begin{enumerate}
\item $d_\alpha$ depends on the set $\{\alpha\}$, 
\item	$d_\alpha[L]=I$.
\end{enumerate}
\end{lemma}
\begin{proof}  By Proposition \ref{gleason}
$\chi_{s_{\overline{Fr}(\kappa)}( e_{\alpha, n})}=
\chi_{s_{\overline{Fr}(\{\alpha\})}( e_{\alpha, n})}\circ p_{\{\alpha\}}$.
So, if $s, t\in L$ satisfy $p_{\{\alpha\}}(s)=p_{\{\alpha\}}(t)$,
then $d_\alpha(s)=d_\alpha(t)$, as required.
The second part follows from the fact that $e_{\alpha, n}$s
as free generators are independent, which implies that 
for every $\sigma\in \{0,1\}^\N$ there is $t\in L$ such that
$\chi_{s_{\overline{Fr}(\kappa)}( e_{\alpha, n})}(t)=\sigma(n)$ for every $n\in \N$. Now use the standard fact
that the mapping $\phi$ from $\{0,1\}^\N$ into $I$ given by $\phi(\sigma)=
\sum_{n\in\N} {\sigma(n)\over{2^n}}$ is surjective.
\end{proof}

\begin{lemma} [\cite{grande}, Lemma 2.10] \label{depends}
  Each $f\in C(L)$ depends on some countable $A\subseteq\kappa$.
\end{lemma}

\begin{lemma} \label{depends-countable}
Suppose that $X, Y$ are compact spaces, $\phi: X\rightarrow Y$ is a continuous
surjection and that $f\in C(X)$ is such that $f(x_1)=f(x_2)$ whenever $\phi(x_1)=\phi(x_2)$.
Then there is a continuous $g\in C(Y)$ such that $f=g\circ \phi$. In particular,
  for every $f\in C(L)$ which depends
on some $A\subseteq \kappa$ there exist  $g \in C(L_A)$ such that $f=g\circ p_A$.
\end{lemma}
\begin{proof}
By the hypothesis one can well define $g: Y\rightarrow \R$ satisfying $f=g\circ \phi$.
Since $\phi$ is
a closed onto mapping (2.4.8. of \cite{engelking}) it is a quotient map
and so $g$ is continuous (2.4.2. of \cite{engelking}).
\end{proof}
\begin{lemma}\label{closed}Suppose that $X, Y$ are compact spaces $\phi: X\rightarrow Y$ is  continuous
and $Z\subseteq X$, then $f[\overline Z]=\overline{f[Z]}$.
\end{lemma}
\begin{proof} Continuous mappings of compact spaces are closed, so use 1.4. C of \cite{engelking}.
\end{proof}

\subsection{Algebras of functions and their Gelfand spaces}

Given $\F\subseteq C_I(L)$ we will consider the closed algebra over the reals containing
constant functions generated by
$\F$ in $C(L)$, that is the real unital Banach algebra generated by $\F$, we will denote it
by $[\mathcal F]$.   $[\F]_I$ will stand for $[\F]\cap C_I(L)$

The role of the Stone space for Boolean algebras is played for commutative Banach algebras by 
the Gelfand space. We will work with the following concrete representation $\nabla\F$
of  the Gelfand space of $[\mathcal F]$ (cf. \cite{grande}):

\begin{definition}[\cite{grande}]
Let $\mathcal{F, G}$ be  families of elements of $C_I(L)$.
  \begin{enumerate}[(1)]

	\item  define $\Pi \mathcal F\colon L \rightarrow I^\mathcal F$ by a formula
	  \[\big((\Pi\mathcal F)(x)\big)(f) = f(x)\]
 for all $x\in L$ and for all $f\in\mathcal F$, 

	\item the image $\Pi\mathcal F[L]\subseteq I^\mathcal F$ is denoted by $\nabla \mathcal F$,

	\item for $\mathcal G \subseteq \mathcal F \subseteq C_I(L)$ we define the natural projection 
	  \[ \pi_{\mathcal G,\mathcal F}\colon \nabla \mathcal F \rightarrow \nabla \mathcal G, \]
which is the restriction of the natural projection from $I^\mathcal F$ to $I^\mathcal G$,
\item given $\F\subseteq C_I(L)$ we say that $f\in C(\nabla\F)$ depends on a set $A\subseteq \kappa$ 
if $f\circ\Pi\F$ depends on $A$.
  \end{enumerate}
\end{definition}

\begin{proposition}\label{gelfand-isomorphism} Suppose that $\F\subseteq C_I(L)$. Then
there is an isometric isomorphism of real Banach algebras $T_{\F}:C(\nabla\F)\rightarrow [\F]$
induced by $\Pi\F$  such that $T_\F(\pi_{\{f\}, \F})=f$ for each $f\in \F$.
\end{proposition}
\begin{proof} The surjective continuous function $\Pi\F: L\rightarrow \nabla F$
induces an isometric isomorphic embedding of $C(\nabla\F)$ into $C(L)$ simply
by sending $g\in C(\nabla\F)$ to $g\circ \Pi\F$. Then it is clear that $T_\F(\pi_{\{f\}, \F})=f$,
so the image of $T_{\F}$ includes $[\F]$.
It remains to show that it is included in $[\F]$. For this it is enough to
show that $C(\nabla\F)$ is generated as a unital algebra by the functions $\pi_{\{f\}, \F}$
for $f\in \F$. This follows from the real Stone-Weierstrass theorem, as the coordinates
separate points in products.
\end{proof}

\begin{lemma}\label{factorization} Suppose that $f\in [\F]\subseteq C_I(L)$. Then there is 
a unique $f(\F)$ in $C(\nabla\F)$
satisfying 
\[f=f(\F)\circ \Pi\F.\]
$f(\F)$ will be called the  factorization
of $f$ through $\F$. If $\F\subseteq \G\subseteq C_I(L)$, then 
$f(\F)\circ\pi_{\F, \G}=f(\G)$.
\end{lemma}
\begin{proof}
Put $f(\F)$ to be $T_{\F}^{-1}(f)$ where $T_{\F}$ is the isometry from Proposition \ref{gelfand-isomorphism}.
The second part follows from the fact that $\pi_{\F, \G}\circ \Pi\G=\Pi\F$
and the uniqueness of the factorization. 
\end{proof}

\begin{lemma}\label{D_0}
Suppose that $\alpha\in \kappa$, then
\begin{enumerate}
\item if $\F$ depends on $\kappa\setminus\{\alpha\}$,
then there is a homeomorphism $\phi: \nabla(\F\cup\{d_\alpha\})\rightarrow (\nabla\F)\times I$
such that $\pi\circ\phi=\pi_{\F, \F\cup\{d_\alpha\}}$ where $\pi$ is the natural
projection from $(\nabla\F)\times I$ onto $\nabla\F$,
\item   $\nabla\D = I^\kappa$.
\end{enumerate}
\end{lemma}
\begin{proof}
 Fix $x\in I$ and $y\in \nabla\F$.
  Using Lemma \ref{d_alpha} let $t\in L$ be such that $d_\alpha(t)=x$.
Fix $s\in L$ such that $(\Pi\F)(s)=y$. Use Lemma \ref{extending-ultrafilters}
to find $u\in L$ such that $p_{\{\alpha\}}(u)=p_{\{\alpha\}}(t)$
and $p_{\kappa\setminus\{\alpha\}}(u)=p_{\kappa\setminus\{\alpha\}}(s)$.
Since $d_\alpha$ depends on $\alpha$ by Lemma \ref{d_alpha} and $\F$ depends on $\kappa\setminus\{\alpha\}$
by the hypothesis, we obtain that 
 $(\Pi(\F\cup\{d_\alpha\}))(u)=(y, x)$ which completes 
the proof of part (1).
(2) follows from (1) applied inductively and Lemma \ref{d_alpha}.
\end{proof}

For a Banach space $X$ a density character of $X$ is a cardinality
 of a minimal dense subset of $X$ and it is denoted by $d(X)$.
\begin{lemma}[GCH] \label{density-cardinality}
  Let $\D\subseteq\F\subseteq C_I(L)$. Then the density character of $C(\nabla\F)$ equals $\kappa$.
\end{lemma}
\begin{proof}
  Using the surjections $\pi_{\D,\F}$ and $\Pi\F$ we obtain isometric injection
 of $C(\nabla\D)$ into $C(\nabla\F)$ and $C(\nabla\F)$ into $C(L)$, so $\kappa\leq d(C(\nabla\F))\leq d(C(L))$
by Lemma \ref{D_0}.
On the other hand, by the Stone-Weierstrass theorem, the density of the Banach space
$C(L)$ is not bigger than the cardinality of the Boolean algebra 
of clopen subsets of $L$ which is isomorphic to the algebra  $\overline{Fr}(\kappa)$,
which is c.c.c. (Proposition \ref{gleason}) and contains a dense subalgebra
${Fr}(\kappa)$ of cardinality $\kappa$. So each element of $\overline{Fr}(\kappa)$
is the supremum of a countable subset of ${Fr}(\kappa)$, hence $|\overline{Fr}(\kappa)|\leq\kappa^\omega$.
So by Lemma \ref{cardinal}, we obtain $d(C(L))\leq\kappa$ which completes the proof.
\end{proof}

It turns out to be convenient to talk about open subsets 
of the Gelfand spaces $\nabla\F$ of the algebras $[\F]$ 
using a language purely depending on $\F$. The next definition
is aiming at this purpose.

\begin{definition}\label{definition-openF}
Suppose that  $\F\subseteq \G \subseteq C_I(L)$. Let $\mathcal J$ denote
the family of all nonempty open subintervals of $[0,1]$ with rational endpoints.
By $\B(\F)$ we denote the family of all partial functions
\[ U: dom(U)\rightarrow \mathcal J,\]
where the domain $dom(U)$ of $U$ is a finite subset of $\mathcal F$.
We will consider the evaluation $U(\G)$ of $U$ at $\G$ which is defined as
\[U(\G)=\{x\in \nabla G: x(f)\in U(f)\ \hbox{for all}\  f\in dom(U)\}.\]
By $U(L)$ we will mean the set
$\{x\in L: f(x)\in U(f)\ \hbox{for all}\  f\in dom(U)\}=\bigcap_{f\in dom(U)}f^{-1}[U(f)]$.
\end{definition}

Note that with the above notation $U(f)$ is the same subinterval of $I$ as the
one defined as $U(\{f\})$.

\begin{lemma}\label{openF}Suppose that  $\F\subseteq \G \subseteq C_I(L)$
and $U, V\in \B(\F)$.
\begin{enumerate}
\item The family of  all sets of the form $W(\F)$ for $W\in \B(\F)$ forms  a basis
of open sets for $\nabla \F$.
\item $(\Pi\F)^{-1}[U(\F)]=U(L)$
\item $\Pi\F[U(L)]=U(\F)$,
\item $\pi_{\F, \G}^{-1}[U(\F)]=U(\G)$,
\item $\pi_{\F, \G}[U(\G)]=U(\F)$
\item $U(\F)\cap V(\F)=\emptyset$,
if and only if $U(\G)\cap V(\G)=\emptyset$.
\item   $U(\F)\subseteq V(\F)$,
if and only if $U(\G)\subseteq V(\G)$.
\item  $\overline{U(\F)}\subseteq V(\F)$,
if and only if $\overline{U(\G)}\subseteq V(\G)$, where the closures are taken in 
$\nabla \F$ and $\nabla \G$ respectively,
\end{enumerate}
\end{lemma}
\begin{proof} The first item is clear from the definition of the product topology.
Item (2) follows directly from Definition \ref{definition-openF}.
Item (3) is the immediate consequence of (2). Item (4)
follows from the fact that $\pi_{\F, \G}^{-1}[X]=\Pi\G[(\Pi\F)^{-1}[X]$ for any $X\subseteq \nabla\F$ and
(2) - (3), namely $\pi_{\F, \G}^{-1}[U(\F)]=\Pi\G[(\Pi\F)^{-1}[U(\F)]=\Pi\G[U(L)]=U(\G)$.
Item (5) is the immediate consequence of (4).  Items (6) - (7)
are the immediate consequences of (4) and the properties of the preimages
of functions. For the forward direction of (8), note 
that always $\overline{\pi_{\F, \G}^{-1}[U(\F)]}\subseteq
 \pi_{\F, \G}^{-1}[\overline{U(\F)}]$ and apply (4). For the backward direction of (8), note 
that always $\overline{\pi_{\F, \G}[U(\G)]}=
 \pi_{\F, \G}[\overline{U(\G)}]$  by Lemma \ref{closed} and apply (5). 
\end{proof}

\begin{definition}
Suppose that  $\F\subseteq  C_I(L)$. A family $\mathcal U\subseteq \B(\F)$
is called an antichain if and only if $U(\F)\cap V(\F)=\emptyset$ for
all $U, V\in \mathcal U$.
\end{definition}

We see  by Lemma \ref{openF} that  the property of
being of  antichain does not change if we pass from $\F$ to a bigger $\G$.
Despite of Lemma \ref{openF} a nontrivial interplay between
 properties $U(\F)$s and $U(\G)$s for $\F\subseteq\G\subseteq C_I(L)$
is possible and will actually be at the heart of the difficulties of the main construction
of this paper. For example, as we want $\nabla\F$ to be connected, we would need
$\Pi\F[U]\cap \Pi\F[L\setminus U]\not=\emptyset$ for any clopen $U\subseteq L$,
In fact, the main properties of $\nabla\F$ for the main
construction of $\F$  (Section 5) corresponding to Definition \ref{paracomplicated} (b) are
expressed in terms of 
the nonemptyness of the intersection 
$\overline{\bigcup_{n\in M}U_n(\F)} \cap \overline{\bigcup_{n\in\N\setminus M}U_n(\F)}$
 for some antichain $(U_n)_{n\in \N}\subseteq \B(\F)$ and $M\subseteq \N$, while 
$\overline{\bigcup_{n\in M}U_n(L)} \cap \overline{\bigcup_{n\in\N\setminus M}U_n(L)}$
is always empty, since $L$ is extremally disconnected, which implies that
it is empty if we replace $U_n(L)$s by $U_n(\G)$s for sufficiently big $\G$.
The following lemma is the first of a series of  observations aiming at
developing techniques of increasing the family $\F$ to a bigger $\G$ with
preserving the nonemptyness of the intersections of the closures of unions as above.

\begin{lemma} \label{keep-intersected}
  Let $\F\subseteq C_I(L)$,  and let $(U_n)_{n\in\N} \subseteq \B(\F)$ be
  an antichain. 
  Let $M\subseteq\N$ be such that there exists $x\in\nabla\F$ with
  \[
    x\in \overline{\bigcup_{n\in M}U_n(\F)} \cap \overline{\bigcup_{n\in\N\setminus M}U_n(\F)}.
  \]
  Then  there exist $s,t\in L$ such that
\begin{enumerate}
  \item  $s\in \overline{\bigcup_{n\in M}U_n(L)}$ and $t\in \overline{\bigcup_{n\in\N\setminus M}U_n(L)}$,
   and 
\item $\Pi\F(s)=\Pi\F(t)=x$.
\end{enumerate}
  Moreover if $f\in C_I(L)$ is such that
  \begin{enumerate}
     \item[(3)] $f(s)=f(t)$ for any $s, t$ satisfying (2),
  \end{enumerate}
   then
  \[
    \overline{\bigcup_{n\in M}U_n(\G)} \cap 
\overline{\bigcup_{n\in\N\setminus M}U_n(\G)}\not=\emptyset,
  \] where $\G=\F\cup\{f\}$.
\end{lemma}
\begin{proof}
By Lemma \ref{closed}  and Lemma \ref{openF} (3) we have
	\[ \Pi\F[\overline{\bigcup_{n\in M} U_n(L)}] 
	=   \overline{\Pi\F[\bigcup_{n\in M} U_n(L)]} 
    =   \overline{\bigcup_{n\in M} U_n(\F)}\ni x,\]
so the existence of $s, t$ as in (1) - (2) follows. 
For (3) again
use Lemma \ref{closed}  and Lemma \ref{openF} (3)   to note that
\[
	\Pi(\F\cup\{f\})(s) 
	\in  \Pi(\F\cup\{f\})[\overline{\bigcup_{n\in M} U_n(L)}] 
	=   \overline{\Pi(\F\cup\{f\})[\bigcup_{n\in M} U_n(L)]} 
    =   \overline{\bigcup_{n\in M} U_n(\G)},\]
  and similarly 
	$\Pi(\F\cup\{f\})(t) \in \overline{\bigcup_{n\in \N\setminus M} U_n(\G)}$,
  which finishes the proof, since 
    $\Pi(\F\cup\{f\})(t)=(x,f(t))=(x,f(s))=\Pi(\F\cup\{f\})(s)$  by (2) and by
the hypothesis of  (3).
\end{proof}

\subsection{Adding suprema}

\begin{definition}\label{Delta} Suppose that $K$ is a compact Hausdorff space.
  For a pairwise disjoint sequence $(f_n)_{n\in\N}\subseteq C_I(K)$ and a function $f\in C_I(K)$ we define the set
  \[ \Delta\big(f,(f_n)_{n\in\N}\big) = \{ x\in K \mid f(x) \neq \sum_{n\in\N} f_n(x) \}. \]
\end{definition}

\begin{lemma}[4.1.(a) \cite{few}] \label{suprema-criterium} Let $K$ be a compact
Hausdorff space.
  A function $f\in C_I(K)$ is the supremum in $C(K)$ of 
a pairwise disjoint sequence $(f_n)_{n\in\N}\subseteq C_I(K)$ if and only if the set $\Delta\big(f,(f_n)_{n\in\N}\big)$
  is nowhere dense.
\end{lemma}

\begin{lemma}\label{domain-sup} Let $K$ be a compact
Hausdorff space.
  For a pairwise disjoint sequence $(f_n)_{n\in\N}\subseteq C_I(K)$ the following set 
  \[ 
    D((f_n)_{n\in\N}) = \bigcup\{ U\mid  U \ \hbox{is open in}\  K\ 
 \hbox{and}\  \{n\in\N\colon U\cap\supp(f_n)\not
=\emptyset\}\  \hbox{is finite}\}   
  \]
  is dense open and for $f = \bigvee_{n\in\N} f_n$ in $C(K)$ we have
  \[
	f \restriction D((f_n)_{n\in\N}) = \sum_{n\in\N} f_n \restriction D((f_n)_{n\in\N}).
  \]
\end{lemma}
\begin{proof}
The first part is  the first part of 4.1. (b) of \cite{few}. The second part
follows from the second part of 4.1. (b) which says that $\sum_{n\in\N} f_n$
is continuous on the open set $D((f_n)_{n\in\N})$, from the fact that
two distinct continuous function differ on an open set and from Lemma \ref{suprema-criterium}.
\end{proof}

\begin{lemma} \label{supp-sup} Let $K$ be a compact
Hausdorff space. Suppose that $(f_n)_{n\in\N}\subseteq C_I(K)$ is an antichain
and $f=\bigvee_{n\in\N}f_n$ in $C(K)$.   Then
  \[
	\supp (f) \subseteq \overline{\bigcup_{n\in\N} \supp(f_n)}.
  \]
\end{lemma}
\begin{proof} The set
 \[\supp (f) \setminus \overline{\bigcup_{n\in\N} \supp(f_n)}\subseteq\Delta\big(f,(f_n)_{n\in\N}\big)\]
 is open. If it was nonempty, it would contradict Lemma \ref{suprema-criterium}.
\end{proof}

In general, when passing from $C(\nabla\F)$ to $C(\nabla\G)$
for $\F\subseteq \G\subseteq C_I(L)$ the supremum $f$ of a
pairwise disjoint sequence  $(f_n)_{n\in \N}$ of functions
in $C(\nabla\F)$ may no longer be its supremum in $C(\nabla\G)$, i.e.,
$f\circ\pi_{\F, \G}$ may not be the supremum of  $(f_n\circ\pi_{\F,\G})_{n\in \N}$.
However, if we use the supremum of $(f_n\circ \Pi\F)_{n\in \N}$ in $C(L)$,
this will not happen as stated in the following:

\begin{lemma}\label{definitionsupindes}
Suppose that $\mathcal F\subseteq C_I(L)$ and 
$(f_n)_{n\in \N}\subseteq[\F]$ is a pairwise disjoint sequence of functions,
and let $f=\bigvee_{n\in \N} f_n$ in $C(L)$.
Then  for every $\mathcal G\subseteq C_I(L)$ 
such that $\mathcal F\cup\{f\}\subseteq \mathcal G$ we have that
the factorization $f(\G)$ of $f$ is the supremum of the factorizations
$(f_n(\G))_{n\in \N}$ in 
 $C_I(\nabla \mathcal G)$.
\end{lemma}
\begin{proof} Use 5.11, 5.12 of \cite{grande} and
the isometric isomorphism between $[\F]$ and $\nabla\F$
from Proposition \ref{gelfand-isomorphism} and Lemma \ref{factorization}.
\end{proof}

\begin{lemma}\label{supremum-dependence} Let $A\subseteq\kappa$. Suppose that 
$(f_n)_{n\in \N}\subseteq C_I(L)$ is pairwise disjoint sequence
of functions which  all depend on $A$.
Then the supremum $\bigvee_{n\in \N}f_n$ in $C(L)$ depends on $A$.
\end{lemma}
\begin{proof} Let $f_n'$ be such a function in $C(L_A)$ that $f_n'\circ p_A=f_n$ for $n\in \N$.
Its existence follows from Lemma \ref{depends-countable}.
It is clear that $f_n'$s are pairwise disjoint as well. As 
$L_A$ is extremally disconnected (Proposition \ref{gleason}), we can take the suprema
of bounded sequences in $C(L_A)$. So let $g=\bigvee_{n\in \N}f_n'$ where
the supremum is taken in $C(L_A)$. It is clear that $g\circ p_A$
depends on $A$. So it is enough to prove that $g\circ p_A$ is the supremum
of $\bigvee_{n\in \N}f_n$ in $C(L)$. Recall Definition \ref{Delta} and let $X=\Delta (g, (f_n')_{n\in \N})$.
It is clear that $Y=p_A^{-1}[X]=\Delta (g\circ p_A, (f_n)_{n\in \N})$. So
by lemma \ref{suprema-criterium} it is enough to prove that that preimages of nowhere dense sets
under $p_A$ are nowhere dense, or that images of open sets
under $p_A$ have nonempty interior.

As $Fr(\kappa)$ is a dense subalgebra of $\overline{Fr}(\kappa)$, it is enough
to prove that $p_A[s_{\overline{Fr}(\kappa)}(a)]$ has a nonempty interior in $L_A$
for any $a\in {Fr}(\kappa)$ (see section 2.2.)
But by the independence of the generators of $Fr(\kappa)$ such an $a$
is a finite sum of elements of the form $a'\wedge a''$ where
$a'\in Fr(A)$ and $a''\in Fr(\kappa\setminus A)$. Moreover 
$p_A[s_{\overline{Fr}(\kappa)}(a'\wedge a'')]=s_{\overline{Fr}(A)}(a')$
by the definition of the Stone functor and the fact that any
ultrafilter in $\overline{Fr}(A)$ which contains $a'$ can
be extended to one in $\overline{Fr}(\kappa)$ which contains $a'\wedge a''$.
It follows that images of open sets
under $p_A$ have nonempty interior which completes the proof.
\end{proof}

\begin{lemma} \label{value-supremum}
  Let $\F\subseteq C_I(L)$, $(f_n)_{n\in\N}\subseteq [\F]$ 
be pairwise disjoint, $(\nu_n)_{n\in\N}\subseteq\{\pm 1\}$,
  $(\eta_n)_{n\in\N}\subseteq \kappa$ be increasing with $d_{\eta_n}\in\F$ for all $n\in\N$ 
  and let $f=\bigvee_{n\in \N}(f_nd_{\nu_n,\eta_n})$ in $C_I(L)$.
  Then 
  \[
	f\restriction \bigcup_{n\in\N} \supp (f_n) =
    \sum_{n\in\N}(f_nd_{\nu_n,\eta_n})\restriction \bigcup_{n\in\N} \supp (f_n).
  \]
\end{lemma}
\begin{proof}
Note that $\supp (f_n)\subseteq D((f_nd_{\nu_n,\eta_n})_{n\in \N})$ and
so the second part of the Lemma \ref{domain-sup} may be used.
\end{proof}

\subsection{Extensions and preserving the connectedness}
For sets $X,Y$ and a function $f\colon X \rightarrow Y$ let us denote the graph of $f$ by $\Gamma(f)$.

\begin{definition}[4.2. \cite{few}]
  Let $K$ be a compact Hausdorff space and let $(f_n)_{n\in\N}\subseteq C_I(K)$ be 
pairwise disjoint. Then the closure of 
  $\Gamma\big( \sum_{n\in\N} f_n\restriction D( (f_n)_{n\in\N})\big)$ in $K\times I$ is called an extension of $K$ 
  by $(f_n)_{n\in\N}$. We denote it by $\ext(K,(f_n)_{n\in\N})$.
  The extension is called strong if $\Gamma\big( \sum_{n\in\N} f_n \big) \subseteq \ext(K,(f_n)_{n\in\N})$.
If $K=\nabla\F$ for some $\F\subseteq C_I(K)$ and $(f_n)_{n\in \N}\subseteq [\F]$,
 then an extension of $\nabla\F$ by $(f_n)_{n\in \N}$ means the
extension of $\nabla \F$ by $(f_n(\F))_{n\in \N}$ and is denoted $\ext(\nabla\F,(f_n)_{n\in\N})$
\end{definition}

Indiscriminate adding of suprema leads to a complete lattice $C(K)$ and implies
that $K$ is extremally disconnected, so
in general extensions of compact spaces do not need to preserve the
 connectedness (for explicite analysis  of this phenomenon in the case of pairwise disjoint
sequences of functions  see \cite{andre}), however
we have the following:

\begin{lemma}[\cite{few}, 4.4. ] \label{graph-connected}
  Let $K$ be a compact and connected Hausdorff space and let $(f_n)_{n\in\N}\subseteq C_I(K)$ be pairwise disjoint. 
The strong extension of $K$ by $(f_n)_{n\in\N}$ is a  compact and connected space. 
\end{lemma}

\begin{lemma} \label{extension-image}
  Let $ \F\subseteq C_I(L)$ and $A\subseteq \kappa$ be such that the family $\F$ depends on 
  $A$ and $\{d_\alpha\mid \alpha\in A\}\subseteq\F$. 
	Let $(f_n)_{n\in\N}\subseteq [\F]$ be pairwise disjoint and let 
	$f\in C_I(L)$ be the supremum of  $(f_n)_{n\in\N}$.
	Then $\ext(\nabla\F,(f_n)_{n\in\N})=\nabla(\F\cup\{f\})$.
\end{lemma}
\begin{proof} Use 5.13. of \cite{grande} and Lemma \ref{factorization}.
\end{proof}

\begin{lemma}[\cite{grande}, 2.7.]\label{finite-connected}
  Let $\F\subseteq C(L)$. Then $\nabla \F$ is connected if and only if 
$\nabla F$ is connected for all finite $F\subseteq \F$.
\end{lemma}

\begin{lemma} \label{strong-extension}
Let
\begin{enumerate}
\item  $\F \subseteq C_I(L)$, 
\item $(f_n)_{n\in\N}\subseteq [\F]_I$ be an pairwise disjoint, 
\item  $\F$ depends on $A\subseteq\kappa$, 
\item $(\eta_n)_{n\in\N}\subseteq \kappa$ be
  such that the set $\{n\in\N \mid \eta_n\in A\}$ is finite, 
\item $(\nu_n)_{n\in\N}\in\{-1,1\}^\N$.
\end{enumerate}
  Then for all infinite $M \subseteq \N$ and for $\G=\F\cup\{d_{\eta_n}\mid n\in\N\}$ the extension of $\nabla\G$ by
  $(f_nd_{\nu_n,\eta_n})_{n\in M}$ is strong.
Moreover,
  if $\nabla\G$ is compact and connected,
 then $\ext(\nabla\G,(f_nd_{\nu_n,\eta_n})_{n\in M})$ is compact and connected as well.
\end{lemma}
\begin{proof} 
Fix some infinite $M\subseteq\N$ and $(x,s) \in \Gamma(\sum_{n\in M} (f_nd_{\nu_n,\eta_n})(\G))$. 
  We need to check that
  \[
	(x,s) \in \overline{\Gamma\big(\sum_{n\in M}  (f_nd_{\nu_n,\eta_n})(\G)
	\restriction D( ((f_nd_{\nu_n,\eta_n})(\G))_{n\in M} ) \big)}.
  \]
  If $s>0$, then $x\in  D( ((f_nd_{\nu_n,\eta_n})(\G))_{n\in M} )$, so we can assume that $s=0$.
  Fix a neighborhood of $(x,0)$ of
the form $U(\G) \times (-\varepsilon,\varepsilon)$  where $U \in 
\B(\G)$ and $\varepsilon > 0$.
  It will be sufficient to show that
  \[
	\left( U(\G) \times (-\varepsilon,\varepsilon) \right)
	\cap \Gamma \big(\sum_{n\in M}  (f_nd_{\nu_n,\eta_n})(\G)
	\restriction D( ((f_nd_{\nu_n,\eta_n})(\G))_{n\in M} ) \big) \neq \emptyset.
  \]
  If $U(\G)$ intersects only finitely many sets $\supp((f_nd_{\nu_n,\eta_n})(\G))$, then 
  we have the inclusion $U(\G)\subseteq D(((f_nd_{\nu_n,\eta_n})(\G))_{n\in M} )$ and  
  so the point $(x,0)$ belongs to the graph $\Gamma\big(\sum_{n\in M} 
(f_nd_{\nu_n,\eta_n})(\F)\restriction
 D(((f_nd_{\nu_n,\eta_n})(\G))_{n\in M})\big)$.

  If $U(\G)$ intersects $\supp((f_nd_{\nu_n,\eta_n})(\G))$ for infinitely many $n\in M$ then,
  by the hypothesis of the lemma, we can pick a number
 $n_0\in M$ such that $\eta_{n_0}\not\in A\cup B$ where
  $B=\{\eta_n\mid d_{\eta_n} \in dom(U)\}$, and there
 is an $x \in U(\G)\cap\supp((f_{n_0}d_{\nu_{n_0},\eta_{n_0}})(\G))$.
In particular $f_{n_0}(\G)(x)\not=0$.
Let $u\in L$ be such that $\Pi\G(u)=x$. Let $v\in L$ be such that 
\[
    0<d_{\nu_{n_0},\eta_{n_0}}(v)<\varepsilon/f_{n_0}(\G)(x),\leqno (1)
  \] 
which exists by Lemma \ref{d_alpha} (2).
By Lemma \ref{extending-ultrafilters} there is $t\in L$ such that
$p_{A\cup B}(t)=p_{A\cup B}(u)$ and $p_{\{\eta_{n_0}\}}(t)=p_{\{\eta_{n_0}\}}(v)$.
Put $y=\Pi\G(t)$.
It follows that $f(t)=f(u)$ for every $f\in dom(U)$ and so
\[y\in U(\G),\leqno (2)\]
 since $\Pi\G(u)=x\in U(\G)$.  Also 
by the dependence of $\F$ by $A$,  Lemma \ref{factorization} and Lemma \ref{d_alpha} we have
\[f_{n_0}(\G)(y)=f_{n_0}(t)=f_{n_0}(u)=f_{n_0}(\G)(x)\not=0,\leqno (3)\]
 and 
$d_{\nu_{n_0},\eta_{n_0}}(\G)(y)=d_{\nu_{n_0},\eta_{n_0}}(t)=
d_{\nu_{n_0},\eta_{n_0}}(v)$
and so by (1) and (3) we have
 \[0<(f_{n_0}d_{\nu_{n_0},\eta_{n_0}})(\G)(y)<\varepsilon.\leqno (4)\]
The first inequality of (4) implies that $y\in D(((f_nd_{\nu_n,\eta_n})(\G))_{n\in M})$
and so that $(y, f_{n_0}d_{\nu_{n_0},\eta_{n_0}})(\G)(y))$ is in 
$\Gamma \big(\sum_{n\in M}  f_nd_{\nu_n,\eta_n}
	\restriction D( (f_nd_{\nu_n,\eta_n})_{n\in M} )$, 
while (2) and the second inequality of (4) imply that it is in $U(\G)\times (-\varepsilon, \varepsilon)$
which completes the proof of the first part of the lemma.
The moreover part follows from Lemma \ref{graph-connected}.
\end{proof}

\subsection{Separations}

\begin{definition}
  Let $\F\subseteq C_I(L)$. We say that an antichain $(U_n)_{n\in\N}$ 
of elements of $\B(\F)$ is separated along $M$ in $[\F]$ iff
  \[ 
	\overline{\bigcup_{n\in M} U_n(\F)} \cap \overline{\bigcup_{n\not\in M} U_n(\F)} = \emptyset,
 \]
where the closures are taken in $\nabla\F$.
\end{definition}

\begin{lemma} \label{separation-equivalent}
  Let $\F\subseteq C_I(L)$ and let $(U_n)_{n\in\N}$ be an antichain
of elements of $\B(\F)$.
  Then $(U_n)_{n\in\N}$ is separated along $M$ in $[\F]$ if and only if there are elements 
  $\{V_j\mid j\in J\},\{V'_j\mid j\in J'\}\subseteq\B(\F)$, with $J$ and $J'$ finite, such that 
  for $V(\F)=\bigcup_{j\in J} V_j(\F)$ and $V'(\F)=\bigcup_{j\in J'} V'_j(\F)$ and we have
  \[
	\bigcup_{n\in M} U_n(\F) \subseteq V(\F)\subseteq\overline{V(\F)}\subseteq V'(\F)\subseteq 
\nabla\F \setminus \big(\bigcup_{n\not\in M} U_n(\F)\big).
  \]
\end{lemma}
\begin{proof}
  It follows directly from the normality of the compact space $\nabla F$ and
from Lemma \ref{openF} (1).
\end{proof}

Note that by Lemma \ref{openF} the above condition from Lemma \ref{separation-equivalent}
is preserved if we pass from
$\F$ to a bigger $\G$. However if an antichain $(U_n)_{n\in \N}$
is not separated along $M\subseteq \N$ in $[\F]$, it may become
separated along $M\subseteq \N$ in $[\G]$.

\begin{lemma} \label{finite-separation}
  Let $\F\subseteq\G\subseteq C_I(L)$ and let $(U_n)_{n\in\N}\subseteq\B(\F)$ be an antichain,
  $M\subseteq\N$. Then the antichain $(U_n)_{n\in\N}\subseteq \B(\G)$ 
is separated along $M$ in $[\G]$ if and only if it is separated in
  $[\F\cup \mathcal H]$ for some finite $\mathcal H\subseteq\G$.
\end{lemma}
\begin{proof}
  Suppose that $(U_n)_{n\in\N}$ is separated along $M$ in $[\G]$ and let 
  $V,V',\{V_j\mid j\in J\},\{V'_j\mid j\in J'\}$ be as in Lemma \ref{separation-equivalent}.
Let $\mathcal H\subseteq \G$ be a finite set including domains of all $V_j$ for $j\in J$
 and all $V_j'$ for all $j\in J'$. Now use Lemma \ref{openF} (7) and (8).
\end{proof}

\begin{lemma} \label{countable-family}
  Let $\F\subseteq C_I(L)$ be countable and let $(U_n)_{n\in\N}\subseteq \B(\F)$ be an antichain. 
  Then there is a set $M\subseteq\N$ such that the antichain $(U_n(\F))_{n\in\N}$ 
  is not separated along $M$ in $[\F]$.  
\end{lemma}
\begin{proof}
The separation
of   $(U_n)_{n\in\N}$ 
   along $M$ would yield finite sets $\{V_j^M: j\in J\}, \{(V_j')^M: j\in J'\}\subseteq \B(\F)$
as in Lemma \ref{separation-equivalent}.
For distinct $M_1, M_2\subseteq N$, these finite families must be distinct. However
there are continuum many subsets of $\N$ while $\B(\F)$ is countable for countable $\F$.
\end{proof}

\begin{definition}\label{separating}
  Let $\F\subseteq\G\subseteq C_I(L)$ and let $(U_n)_{n\in\N}$ be an antichain in $\B(\F)$.
  We say that $\F$ is separating for $(U_n)_{n\in\N}$ in $\G$, if whenever
 $(U_n(\G))_{n\in\N}$ is separated along some
  $M\subseteq\N$ in $[\G]$, then already $(U_n(\F))_{n\in\N}$ is separated along $M$ in $[\F]$.
\end{definition}

\begin{lemma} \label{adding-disjoint}
  Let $\F\subseteq\G\subseteq C_I(L)$ and let $(U_n)_{n\in\N}$ be an antichain in $\B(\F)$ such that
  $\F$ is separating for $(U_n)_{n\in\N}$ in $\G$. Suppose that $(f_n)_{n\in\N}\subseteq C_I(L)$ is an antichain
  such that $U_n(L)\cap \supp (f_m) = \emptyset$ for every $n,m\in\N$ and let $f\in C_I(L)$ be the supremum of 
  $(f_n)_{n\in\N}$ in $C(L)$. Then $\F$ is separating for $(U_n)_{n\in\N}$ in $\G\cup\{f\}$.
\end{lemma}

\begin{proof}
  Fix $M\subseteq\N$ and suppose that the antichain 
$(U_n(\F))_{n\in\N}$ is not separated along $M$ in $[\F]$. 
  Then the hypothesis of the lemma guarantees that $(U_n(\G))_{n\in\N}$
 is not separated along $M$ in $[\G]$ and
  so we can pick $x\in \overline{\bigcup_{n\in M} U_n(\G)} \cap
 \overline{\bigcup_{n\in \N\setminus M} U_n(\G)}\subseteq \nabla\G$. 
  Now it is enough to find appropriate $s,t\in L$ and use Lemma \ref{keep-intersected}.

  By the hypothesis we have for every $P\subseteq\N$:
  \[ \bigcup_{n\in P} U_n(L) \cap \bigcup_{n\in\N} \supp (f_n)=\emptyset, \]
  but the space $L$ is extremally disconnected, hence for every $P\subseteq\N$ we have
  \[ \overline{\bigcup_{n\in P} U_n(L)} \cap \overline{\bigcup_{n\in\N} \supp (f_n)}=\emptyset. \leqno{(\ast)}\]
By Lemma \ref{keep-intersected} (1) - (2) there are
 $s\in \overline{\bigcup_{n\in M} U_n(L)}$ such that $(\Pi\G)(s)=x$
and   $t\in\overline{\bigcup_{n\in \N\setminus M} U_n(L)}$ such that $(\Pi\G)(t)=x$. 
  By $(\ast)$ we have that $s,t \not\in \overline{\bigcup_{n\in\N} \supp (f_n)}$, so by Lemma \ref{supp-sup} we have $s,t \not\in \supp(f)$, which means exactly that $f(s)=f(t)=0$ 
   so
an application of  Lemma \ref{keep-intersected} (3) completes  the proof.
\end{proof}

\begin{lemma} \label{extending-d_alpha}
  Let $A\subseteq\kappa$,  $\alpha\in \kappa\setminus A$  and suppose  
that $\F\subseteq C_I(L)$  depends on $A$ and  
  $(U_n)_{n\in\N}$ is an antichain in $\B(\F)$.
  Then $\F$ is separating for $(U_n)_{n\in\N}$ in $\F \cup \{d_\alpha\}$.
\end{lemma}
\begin{proof}
By Lemma \ref{D_0} (1) the sets $U_n(\F\cup\{d_\alpha\})$ correspond
to $U_n(\F)\times I$, so the lemma follows.
\end{proof}

\begin{lemma} \label{extending-generator} Suppose that we are given
\begin{enumerate}
\item $A\subseteq \kappa$, 
\item  $\F\subseteq C_I(L)$  which depends on $A$ 
with $\{d_\alpha\mid\alpha\in A\}\subseteq\F$,
\item  an antichain $(U_n)_{n\in\N} \subseteq \B(\F)$ 
\item a pairwise disjoint  $(f_n)_{n\in\N}\subseteq [\F]_I$
\item  $(\nu_n)_{n\in\N}\subseteq\{\pm 1\}$, 
\item   a strictly  increasing $(\eta_n)_{n\in\N}\subseteq \kappa$  with 
the set $\{n\in\N\mid\eta_n\in A\}$  finite.
\end{enumerate}
  Let 
\[f =\bigvee_{n\in\N}f_nd_{\nu_n,\eta_n}\] 
in $C_I(L)$.
  Then $\G$ is separating for $(U_n)_{n\in\N}$ in $\G \cup \{f\}$ where
  $\G=\F\cup\{d_{\eta_n}\mid n\in\N\}$.
\end{lemma}
\begin{proof}
  Fix $M\subseteq\N$ such that the antichain $(U_n)_{n\in\N}$ is not separated along $M$ in $[\G]$,
 we will show that $(U_n)_{n\in\N}$ is not separated along $M$ in $[\G\cup\{f\}]$.
Let $F\subseteq\N$ be the finite set of all $n$'s such that $\eta_n\in A$ and let
$\mathcal H=\F\cup\{\eta_n: n\in F\}\subseteq \G$. Lemma \ref{finite-separation} implies
that $(U_n)_{n\in\N}$ is not separated along $M$ in $[\mathcal H]$.
 Hence there is $x\in \overline{\bigcup_{n\in M} U_n(\mathcal H)} \cap 
\overline{\bigcup_{n\in \N\setminus M} U_n(\mathcal H)}\subseteq \nabla\mathcal H$. 
  Using Lemma \ref{keep-intersected} (1) - (2)  fix $t\in\overline{\bigcup_{n\in M} U_n(L)}$ and 
  $s\in\overline{\bigcup_{n\in \N\setminus M} U_n(L)}$ such that $\Pi\mathcal H(t)=\Pi\mathcal H(s)=x$. 

As $\mathcal H$ depends on $A$
we have $\Pi\mathcal H(t')=\Pi\mathcal H(s')=x$ for any $t',s'\in L$ such that $p_A(t')=p_A(t)$ and
$p_A(s')=p_A(s)$. Using this
observation and inductively Lemma \ref{d_alpha} and  Lemma \ref{extending-ultrafilters}
we may assume that $d_{\nu_n,\eta_n}(s)=d_{\nu_n,\eta_n}(t)=0$ for
all $n\in \N\setminus F$.  Inductive application of Lemma \ref{keep-intersected}
and later Lemma \ref{finite-separation} implies that $(U_n)_{n\in\N}$ is 
not separated along $M$ in $[\G\cup\{g\}]$ where
\[g =\bigvee_{n\in\N\setminus F}f_nd_{\nu_n,\eta_n}\] 
in $C_I(L)$.
So it is enough to show that there is a continuous surjection
$\phi: \nabla(\G\cup\{g\})\rightarrow \nabla(\G\cup\{f\})$ such that
$\phi[U(\G\cup\{g\})]=U(\G\cup\{f\})$ for every $U\in \B(\G)$. 
For this it is enough to have a continuous surjection
$\psi: \nabla(\G\cup\{g\})\rightarrow \nabla(\G\cup\{f, g\})$ such that
$\phi[U(\G\cup\{g\})]=U(\G\cup\{f, g\})$ for every $U\in \B(\F)$
since
then we can consider $\phi=\pi_{\G\cup\{f\}, \G\cup\{f, g\}}\circ\psi$ and Lemma \ref{openF} (5).
To get $\psi$ note that $f$ is the composition
of $\Pi (\G\cup\{g\})$ with the sum $h$ of
two continuous functions on $\nabla(\G\cup\{g\})$ namely $g( \G\cup\{g\})$ and
$\Sigma_{n\in F}((f_nd_{\nu_n,\eta_n})( \G\cup\{g\}))$, so 
$\psi(x)=(x, h(x))$ works.
\end{proof}
\subsection{Butterfly points}

\begin{lemma} \label{antichain-existence}
  Let $\F\subseteq C_I(L)$ and let $U\subseteq\nabla\F$ be open. 
Then there is a countable subset $\F_0\subseteq\F$
  and an antichain $(U_n)_{n\in\N}\subseteq \B(\F_0)$ such that 
  \[
	\overline U = \overline{ \bigcup_{n\in\N} U_n(\F)} \text{ in } \nabla\F.
  \]
\end{lemma}
\begin{proof} Let $\mathcal U$ be a maximal with respect to inclusion
subfamily of  $\B(\F)$ such that 
\begin{itemize}
\item $V(\F)\subseteq U$ for all $V\in \mathcal U$,
\item $V(\F)\cap W(\F)=\emptyset$ for any two distinct $V, W\in \mathcal U$.
\end{itemize}
Note that $\mathcal U$ is countable since $\nabla F$ satisfies the c.c.c. as 
a continuous image of $L$  and by Proposition \ref{gleason}.
The maximality together with
Lemma \ref{openF} (1) gives $\overline {U} = \overline{ \bigcup_{n\in\N} U_n(\F)}$.
\end{proof}

Recall the notion of a butterfly point from Definition \ref{multiplier}.

\begin{lemma} \label{ladder-nobutterfly}
  Let $\D=\{d_\alpha: \alpha<\kappa\}\subseteq \F\subseteq C_I(L)$ be of the form
$\F=\bigcup_{\alpha<\kappa}\F_\alpha$ 
for $\F_\alpha$s
satisfying 
$ \F_{\alpha'}\subseteq \F_{\alpha},$ $\F_{\alpha+1}=\F_\alpha\cup\{d_\alpha\}$,
 $\F_\alpha$ depends on $\alpha$
and $\F_\alpha\cup\{d_\alpha\}$ is separating in $\F$ for every antichain in $\B(\F_{\alpha'})$ for
each $\alpha'<\alpha<\kappa$.
  Then $\nabla\F$ has no butterfly points.
\end{lemma}
\begin{proof}
  Fix two disjoint open set $U, V\subset\nabla\F$ such that there exists $x\in\overline U \cap \overline V$. 
We will show that $\overline U \cap \overline V$ contains at least two distinct points.
  By Lemmas \ref{antichain-existence} and \ref{depends} there exist countable sets  $A\subseteq \kappa$
and $\mathcal G\subseteq \F$ and antichains
  $(U_n)_{n\in\N}\subseteq\B(\mathcal G)$ and $(V_n)_{n\in\N}\subseteq\B(\mathcal G)$ such that
$\mathcal G$ depends on $A$ and
  \[
	\overline U = \overline{ \bigcup_{n\in\N} U_n(\F)} \text{ and }
	\overline V = \overline{ \bigcup_{n\in\N} V_n(\F)} \text{ in } \nabla\F.
  \]
  Using  the regularity of $\kappa$ we see that there 
exists $\alpha<\kappa$ such that 
 $A\subseteq \alpha$ and  $\G\subseteq \F_\alpha$ 

  By Lemma \ref{openF} (6) 
we have that $\pi_{\F_\alpha,\F}[U_n(\F)]=U_n(\F_\alpha)$
and $\pi_{\F_\alpha,\F}[V_n(\F)]=V_n(\F_\alpha)$, so
$x'\in \overline{ \bigcup_{n\in\N} U_n(\F_\alpha)}\cap \overline{ \bigcup_{n\in\N} V_n(\F_\alpha)}$ where
$x'=\pi_{\F_\alpha,\F}(x)$.
It is clear that in $\nabla\F_\alpha\times [0,1]$ we have
$(x', {u+v\over2})\in \overline{ \bigcup_{n\in\N} U_n(\F_\alpha)\times (u,v)}\cap
 \overline{ \bigcup_{n\in\N} V_n(\F_\alpha)\times(u, v)}$ for
any $0<u<v<1$.

Now  define $W_{2n}^{u, v}=U_n\cup\{\left\langle d_\alpha, (u, v)\right\rangle\}$
in  $\B(\F_{\alpha+1})$
 and 
  $W_{2n+1}^{u, v}=V_n\cup\{\left\langle d_\alpha, (u, v)\right\rangle\}$ in  
$\B(\F_{\alpha+1})$ for all $n\in\N$.
By the above observation and Lemma \ref{D_0}  (1),  for every $0<u<v<1$ the sequence
  $(W_n^{u, v})_{n\in\N}$ is an antichain in
 $\B(\F_{\alpha+1})$ which is not
 separated along $2\N$ in
  $\nabla(\F_{\alpha+1})$.  By Lemma \ref{extending-d_alpha}
and the hypothesis that $\F_{\alpha+1}$ depends on $\alpha+1$ the sequence
$(W_n^{u, v})_{n\in\N}$ is an antichain in
 $\B(\F_{\alpha+1})$ which is not
 separated along $2\N$ in
  $\nabla(\F_{\alpha+2})$. 
Hence by the hypothesis
that $\F_{\alpha+2}$ is separating in $\F$ for every antichain in $\F_{\alpha+1}$,
 the antichain $(W_n^{u,v})_{n\in\N}$ 
  is not separated along $2\N$ in $\nabla\F$.

Now let $X=\left\langle d_\alpha, (0, 1/3)\right\rangle$ and 
$Y=\left\langle d_\alpha, (2/3, 1)\right\rangle$ be elements of $\B(\F)$.
We see that $W_n^{0,1/3}(\F)\subseteq X(\F)$ as well as 
$W_n^{2/3, 1}(\F)\subseteq Y(\F)$ while $\overline{X(\F)}\cap \overline{Y(\F)}=\emptyset$
which shows that the points witnessing the nonseparation
of $(W_n^{0, 1/3}(\F))_{n\in\N}$  and $(W_n^{2/3, 1}(\F))_{n\in\N}$
 along $2\N$ in $\nabla\F$ must be distinct.

On the other hand we have $W_{2n}^{u, v}(\F)\subseteq U_n(\F)$
and  $W_{2n+1}^{u, v}(\F)\subseteq V_n(\F)$
for any $0<u<v<1$ and any $n\in \N$ which shows that these two
distinct points must belong to 
$\overline{ \bigcup_{n\in\N} U_n(\F)}\cap \overline{ \bigcup_{n\in\N} V_n(\F)}$
and consequently to $\overline U\cap\overline V$ which completes the proof.
\end{proof}

\section{Ladder  families}

\begin{definition} \label{ladder} 
  Let $\lambda< \kappa$ and $S\subseteq E^\lambda_\omega$. 
  We say that a family $\F\subseteq C_I(L_\lambda)$ is a ladder family of length $\lambda$ given by
the following parameters defined for all $\alpha\in S$:
  \begin{enumerate}
	\item $(\nu^\alpha_n)_{n\in\N} \subseteq \{-1,1\}$,
    \item sequences $(\eta^\alpha_n)_{n\in\N}\subseteq \alpha$ increasingly convergent to $\alpha$,
	\item pairwise disjoint  $(f^\alpha_n)_{n\in\N}\subseteq [\F]$, 
which depends on some $\beta_\alpha< \alpha$,
	\item infinite coinfinite set of integers $M_\alpha\subseteq\N$,
  \end{enumerate}
  if $\F=\{d_\alpha\mid \alpha<\lambda\} \cup\{g_\alpha\mid\alpha\in S\}$ where
  \[
	g_\alpha=\bigvee_{n\in M_\alpha} f_n^\alpha d_{\nu^\alpha_n,\eta^\alpha_n}
  \  \ \hbox{in $C(L)$}\]
  and each $f^\alpha_n$ belongs to 
  \[
	[\{d_\beta\mid\beta<\beta_\alpha\}\cup\{g_\beta\mid\beta\in S\cap {\beta_\alpha}\}].
  \]
  Given $B\subseteq \lambda$ we denote the family $\{d_\alpha \mid \alpha\in B\}$ by $\D[B]$ and 
  the family $\D[B] \cup\{g_\alpha\mid\alpha\in B\cap S\}$ by $\F[B]$.
\end{definition}

Thus a ladder family is a family determined by $S$ and the parameters as in (1) - (4) and constructed
in a recursive manner following the values of these parameters.

\begin{lemma}\label{fragments-dependence} Suppose
that $\F$ is a ladder family of length $\lambda$. Then $\F[\alpha]$ depends on
$\alpha$ for every $\alpha<\lambda$.
\end{lemma}
\begin{proof} Use the recursive definition of
$\F[\alpha]$, Lemma \ref{d_alpha} (1) and Lemma \ref{supremum-dependence}.
\end{proof}

\begin{lemma} \label{separating-ordinal}
  Let $\lambda<\kappa$ and let $\F$ be a ladder family of length $\lambda$ and 
  let $(U_n)_{n\in\N} \subseteq \B(\F[\lambda_0])$ be an antichain for some $\lambda_0<\lambda$. 
  Then the family $\F[\lambda_0\cup\{\lambda_0\}]$ is separating for  $(U_n)_{n\in\N}$ in $\F$.
\end{lemma}
\begin{proof}
  Let  $S\subseteq E^\lambda_\omega$ be as in the definition of a ladder family and let $M\subseteq\N$ 
be such that 
  the antichain $(U_n)_{n\in\N}\subseteq \B(\F[\lambda_0])$ 
is not separated along $M$ in $[\F[\lambda_0\cup\{\lambda_0\}]]$.
  By Lemma \ref{finite-separation} it is enough to show that the 
antichain  $(U_n)_{n\in\N}$ is not separated along $M$ 
  in $[\F[\alpha+1]]$ for
 all $\alpha \in [\lambda_0,\lambda)$. We proceed by induction.
  The base step $\alpha=\lambda_0$ follows from the choice of $M$.
  For the inductive step we fix $\alpha\in  (\lambda_0,\lambda)$ 
and assume the hypothesis for all ordinals 
  smaller than $\alpha$. 
  Then by Lemma \ref{finite-separation} the antichain $(U_n)_{n\in\N}$ is not separated along 
$M$ in $\F[\alpha]$ and we have the following two cases:

  \noindent Case 1. $\alpha\not\in S$. Then $\F[\alpha+1]=\F[\alpha]\cup\{d_\alpha\}$ and 
  the family $\F[\alpha]$ depends on $\alpha$ so we can use Lemma \ref{extending-d_alpha}. 

  \noindent Case 2. $\alpha\in S$. Then $\F[\alpha+1]=\F[\alpha]\cup\{g_\alpha,d_\alpha\}$.
  First, we show that the antichain  $(U_n)_{n\in\N}$ is not separated along $M$ in 
$[\F[\alpha]\cup\{g_\alpha\}]$.
  By the definition of the ladder family we have that
  \(
	g_\alpha=\bigvee_{n\in M_\alpha} ( f^\alpha_nd_{\nu^\alpha_n,\eta^\alpha_n})
  \)
  and all $f^\alpha_n$s are from $[\F[\beta_\alpha]]$.
  Now, as $\beta_\alpha < \alpha$, we observe that by Lemma \ref{finite-separation}, it is enough to show
  that the antichain $(U_n)_{n\in\N}$ is not separated along $M$ in $[\F[\beta]\cup\{g_\alpha\}]$ for all
  $\beta \in [\beta_\alpha,\alpha)$. 
  But this follows from Lemma \ref{extending-generator} for
  $\F=\F[\beta+1]$, $A=\beta$, $f=g_\alpha$ and the inductive hypothesis that for all 
  $\beta \in [\beta_\alpha,\alpha)$ the antichain $(U_n)_{n\in\N}$ is not separated 
  along $M$ in $[\F[\beta+1]]$.

  Finally, to conclude that the antichain $(U_n)_{n\in\N}$ is not separated along $M$ in 
  $[\F[\alpha]\cup\{g_\alpha,d_\alpha\}]$ we use Lemma \ref{extending-d_alpha} 
as in the Case 1 and Lemma \ref{supremum-dependence}. 
  The proof of the Lemma is finished.
\end{proof}

\begin{lemma}\label{monolithic} 
  Suppose that $\kappa$ is an uncountable regular cardinal, $S\subseteq E^\kappa_\omega$ and a strictly
  increasing sequence 
  $(\eta^\alpha_n)_{n\in\N}\subseteq \kappa$ is convergent to $\alpha$ for every $\alpha\in S$. 
  Then for every countable $A\subseteq \kappa$ the set $S_A$ of all $\alpha\in S$ such that 
  $\{\eta^\alpha_n\mid n\in\N\} \cap A$ is infinite is at most countable.
\end{lemma}
\begin{proof} 
  Define $f:S_A\rightarrow A\cup\{\sup A\}$ by putting for $\alpha\in S_A$
  the value $f(\alpha)$ to be the least upper bound of the set $\{\eta^\alpha_n\mid n\in\N\} \cap A$ 
among the elements of the set $A\cup\{\sup A\}$.
  Since $A\cup\{\sup A\}$ is countable it is enough to check the injectivity of $f$.
  Fix $\alpha, \alpha'\in S_A$ such that $\alpha<\alpha'$. $A\cap\{\eta^{\alpha'}_n\mid n\in\N\}$ is cofinal in $\alpha'$, so we can pick $n_0$ such that 
  $\alpha < \eta^{\alpha'}_{n_0}$ and $\eta^{\alpha'}_{n_0} \in A$. 
  Then we see that $f(\alpha) \leq \eta^{\alpha'}_{n_0} < f(\alpha')$.
\end{proof}

\begin{lemma} \label{separating-countable}
  Let $\lambda<\kappa$, let $\F$ be a ladder family of length $\lambda$ and 
  let $(U_n)_{n\in\N} \subseteq \B(\F)$ be an antichain. 
  Then there is a countable $A\subseteq \lambda$ such that $(U_n)_{n\in\N}\subseteq\B(\F[A])$
  and $\F[A]$ is separating for $(U_n)_{n\in\N}$ in $\F$.
\end{lemma}
\begin{proof}
  Fix the set $S$ from the definition of a ladder family.
  \begin{claim}
	There exists an increasing sequence $(A_n)_{n\in\N}$ of countable subsets of $\lambda$ such that
	\begin{enumerate}[(i)]
	  \item $(U_n)_{n\in\N} \subseteq \B(\F[A_0])$,
	  \item for all $n\in\N$ the family $\F[A_n]$ depends on $A_n$,
	  \item for all $\alpha\in S$ and all $n\in\N$ if $\eta^\alpha_k\in A_n$ for infinitely many $k$ then $\alpha\in A_{n+1}$.
	\end{enumerate}
  \end{claim}
  \begin{proof}
	By the assumption the domain
of every $U_n$ is  some finite set $F_n\subseteq\F$ of coordinates.
	Fix an arbitrary countable set $A_{0,0}\subseteq \lambda$ 
such that $\bigcup_{n\in\N} F_n \subseteq \F[A_{0,0}]$.  
	This choice guarantees that $(U_n)_{n\in\N} \subseteq \B(\F[A_{0,0}])$.
	Then for every $n\in\N$ define 
$A_{0,n+1}\subseteq \lambda$ as the union of $A_{0,n}$ and the countable set 
	$\bigcup\{ Y_\alpha\mid \alpha\in A_{n,0}\cap S\}$ where 
$Y_\alpha\subseteq \alpha$ is some countable set such that 
	$g_\alpha$ depends on $Y_\alpha$ for  $\alpha\in S$. 
	We see that $A_{0,n}$ is countable for every $n\in\N$ and so is the set
	$A_0=\bigcup_{n\in\N} A_{0,n}$. Now we have $(i)$ and $(ii)$ for $n=0$.

	Fix $n\in\N$ and assume we have defined $A_n$ such that
 $(ii)$ and $(iii)$ hold. We define $A_{n+1}$ in two steps. 
	First, we use Lemma \ref{monolithic} with $S=S$ and $A=A_n$ to 
obtain the countable set $S_{A_n}$ so that 
	we know the set $A_{n+1,0}=A_n\cup S_{A_n}$ is countable and that any
superset $A_{n+1}$ of $A_{n+1,0}$ satisfies (iii).
	Then we apply the procedure outlined above for constructing $A_0$ 
	to obtain countable $A_{n+1}$ such that $(i)$ and $(ii)$ hold. This completes
the proof of the Claim.
  \end{proof}
  Fix a sequence $(A_n)_{n\in\N}$ from the above claim and set $A=\bigcup_{n\in\N} A_n$. 
  We will show that $\F[A]$ is separating for $(U_n)_{n\in\N}$ in $\F$. 
  So suppose  that $(U_n)_{n\in\N}$ is not separated along $M$ in $[\F[A]]$.
  By Lemma \ref{finite-separation} it is enough to show that $(U_n)_{n\in\N}$ is not separated along $M$
  in $[\F[A\cup \alpha+1]]$ for every $\alpha<\lambda$.
  We prove it by induction on $\alpha<\lambda$. 
  The base step for $\alpha=0$ that $(U_n)_{n\in\N}$ is not separated along $M$ in $[\F[A]]$ 
follows from the choice of $M$.
  Now assume that $\alpha<\lambda$ and that the hypothesis is 
true for all ordinals smaller than $\alpha$ that is
  the antichain $(U_n)_{n\in\N}$ is not separated along $M$ in $[\F[A\cup \alpha]]$.
  We have the following three cases:
  
  \noindent \emph{Case} 1. $\alpha\in A$. 
  Then $\F[A\cup (\alpha+1)]=\F[A \cup\alpha\cup\{\alpha\}]=\F[A\cup\alpha]$ and
 we are done by inductive hypothesis.
  
  \noindent \emph{Case} 2. $\alpha\not\in A$ and $\alpha\not\in S$. 
  Then 
  \[
	\F[A\cup (\alpha+1)]=\F[A\cup \alpha]\cup\{d_\alpha\}.
  \]
  The family $\F[A\cup \alpha]$ depends on $A\cup \alpha$ by $(ii)$ of the above claim, by Lemma 
  \ref{fragments-dependence} and by definition of ladder family.
  Therefore, we can use Lemma \ref{extending-d_alpha}   to
  conclude that the antichain $(U_n)_{n\in\N}$ is not separated along $M$ in $[\F[A\cup (\alpha+1)]]$.

  \noindent \emph{Case} 3. $\alpha\not\in A$ and $\alpha\in S$.
  Then 
  \[
	\F[A\cup (\alpha+1)]=\F[A\cup \alpha]\cup\{g_\alpha,d_\alpha\}.
  \]
  We prove the inductive step in this case in two steps. 
  In the first step we show that the antichain $(U_n)_{n\in\N}$ is not separated along $M$ in 
  $[\F[A_k\cup \alpha]\cup\{g_\alpha\}]$ for every $k\in\N$, which, by Lemma 
\ref{finite-separation}, implies that 
  the antichain $(U_n)_{n\in\N}$ is not separated along $M$ in $[\F[A\cup \alpha]\cup\{g_\alpha\}]$.
  Fix $k\in\N$ and let $\beta_\alpha<\alpha$ be as in the definition of a ladder family.
  Using Lemma \ref{finite-separation} once more we see that it is sufficient to show that the antichain $(U_n)_{n\in\N}$ is 
  not separated along $M$ in $[\F[A_k\cup \beta]\cup\{g_\alpha\}]$ for all $\beta\in(\beta_\alpha,\alpha)$.
  Fix $\beta\in(\beta_\alpha,\alpha)$ and apply Lemma \ref{extending-generator} with $\F=\F[A_k\cup \beta]$,
  $A=A_k\cup \beta$ and $f=g_\alpha$.
  Let us check the assumptions of Lemma \ref{extending-generator}:
  \begin{itemize}
	\item the family $\F[A_k\cup \beta]$ depends on the set 
	  $A_k\cup \beta$ by $(ii)$ of the above claim, Lemma \ref{fragments-dependence} 
	  and the definition of ladder family,
	\item we have $\D[A_k\cup \beta]\subseteq \F[A_k\cup \beta]$ by the definition of operation $\F[(\cdot)]$,
	\item the elements $f^\alpha_n$ all depend on $\beta$ because we have $\beta_\alpha < \beta$,
	\item the set $\{n\in\N\mid\eta^\alpha_n\in A_k\cup \beta\}$ is finite because 
	  by the assumption of this case $\alpha\not\in A_{k+1}$ and  $(\eta^\alpha_n)_{n\in\N}$
 increasingly converges to $\alpha>\beta$.
  \end{itemize}
  As the antichain $(U_n)_{n\in\N}$ is not separated along $M$ in $[\F[A_k\cup \beta]]$, then Lemma 
  \ref{extending-generator} guarantees that the antichain $(U_n)_{n\in\N}$ is not separated along $M$ in 
  $[\F[A\cup \beta]\cup\{g_\alpha\}]$.

  The second step of the proof of Case 3 consists of showing that the 
antichain $(U_n)_{n\in\N}$ is not separated along $M$ in 
  $[\F[A\cup \alpha]\cup\{g_\alpha,d_\alpha\}]$. This is done by Lemma \ref{extending-d_alpha} since
  the family $\F[A\cup \alpha]\cup\{g_\alpha\}$ depends on 
  $ A\cup \alpha$. This completes the inductive step and hence the proof of the lemma.
\end{proof}

\begin{lemma} \label{ladder-connected}
  Let $\F$ be a ladder family of length $\kappa$. 
  Then $\nabla\F$ is connected.
\end{lemma}
\begin{proof}
  By Lemma \ref{finite-connected} it is enough to show that
  $\nabla\F[\alpha+1]$ is connected for all $\alpha<\kappa$. 
  We use the transfinite induction so let us fix $\alpha<\kappa$ and let us assume that we are done below 
  $\alpha$. The inductive hypothesis implies that $\nabla\F[\alpha]$ is connected. 
  If $\alpha\not\in S$ then 
  \[
	\nabla(\F[\alpha+1])=\nabla(\F[\alpha]\cup\{d_\alpha\})=\nabla(\F[\alpha])\times I
  \]
  by Lemma \ref{D_0}, since $\F[\alpha]$ depends on $\alpha$ by Lemma \ref{fragments-dependence}, so we are done. 
  If $\alpha\in S$ then $\F[\alpha+1]=\F[\alpha]\cup\{g_\alpha,d_\alpha\}$.
  By the definition of the ladder family we have
  \[
	g_\alpha=\bigvee_{n\in M_\alpha}f^\alpha_nd_{\nu^\alpha_n,\eta^\alpha_n} 
  \]
  and all $f^\alpha_n$s are from $[\F[\beta_\alpha]]$ where
$\beta_\alpha<\alpha$. 
  Using Lemma \ref{finite-connected} once more we see that it is enough to prove that 
  $\nabla(\F[\beta]\cup\{g_\alpha\})$ is connected for all $\beta\in(\beta_\alpha,\alpha)$.
  By Lemma \ref{extension-image} we have that
  \[
	\nabla(\F[\beta]\cup\{g_\alpha\}) = 
	\ext\big(\nabla(\F[\beta]), (f^\alpha_nd_{\nu^\alpha_n,\eta^\alpha_n})_{n\in\N}\big)
  \]
  and we use Lemma \ref{strong-extension} to conclude that $\nabla(\F[\beta]\cup\{g_\alpha\})$ 
  is connected for all $\beta\in(\beta_\alpha,\alpha)$ since $\nabla(\F[\beta])$
is connected by the inductive hypothesis.
With adding $d_\alpha$ we proceed as in the first case.
\end{proof}

\begin{lemma}\label{ladder2-nobutterfly}  Let $\F$ be a ladder family of length $\kappa$. 
  Then $\nabla\F$ has no butterfly points.
\end{lemma}
\begin{proof} Put $\F_\alpha=\F[\alpha+1]$ and note that the hypothesis
of Lemma \ref{ladder-nobutterfly} is satisfied by Lemmas \ref{fragments-dependence}
and \ref{separating-ordinal}, so the proof is completed by applying Lemma \ref{ladder-nobutterfly}.
\end{proof}

\section{The construction}

\begin{lemma}\label{cardinal} Assume the GCH.
Let  $\kappa$ be a regular cardinal which is of the form $\lambda^+$ for $\lambda$ which is a cardinal of uncountable cofinality.
Then $\kappa^\omega=\kappa$ and for every $\alpha<\kappa$ we have \[\alpha^\omega<\kappa\]
\end{lemma}
\begin{proof}
We prove the lemma by induction on a cardinal $\alpha<\kappa$.
If $cf(\alpha)=\omega$, then  $\alpha^\omega\leq 2^\alpha=\alpha^+\leq\kappa$ by the GCH and 
$\alpha^+<\kappa$ by the hypothesis on $\kappa$. If $cf(\alpha)>\omega$, then
$\alpha^\omega=\sum\{\beta^\omega\mid \beta<\alpha\}$ which is less than $\kappa$ 
by the inductive assumption and the regularity of $\kappa$.
It also follows that $\kappa^\omega=\kappa$.
\end{proof}

When talking about topological concepts like convergence
in the context of ordinals we always refer to the order topology on the ordinals.
Recall that a subset $C\subseteq \kappa$ is called club if and only if it is unbounded in $\kappa$
and closed in the order topology. $S\subseteq \kappa$ is called stationary if it intersects
all club sets. It is well known that $E^\kappa_\omega$ is stationary for any uncountable
regular $\kappa$ (see \cite{jech}, \cite{kunen}). Assuming GCH we have the following
theorem due to Gregory:

\begin{theorem}[Theorem $23.2$ \cite{jech}]\label{diamond} Assume GCH.
  There is a sequence $(S_\alpha)_{\alpha\in E^\kappa_\omega}$ such that:
  \begin{enumerate}
	\item $S_\alpha\subseteq \alpha$ for every $\alpha\in E^\kappa_\omega$,
	\item for every $X\subseteq \kappa$ the set
	  \[
		\{\alpha\in E^\kappa_\omega \mid X\cap \alpha=S_\alpha\}
	  \]
	is stationary in $\kappa$.
\end{enumerate}
The above statement is called $\diamondsuit(E^\kappa_\omega)$. 
\end{theorem}

Fix a bijection $\Psi$ from $\kappa$ onto
$
  (\{-2, -1\}\cup\kappa)\times \{-1, 1\}^\N \times 
  \big( C_I(L)^\N \cup \B(C_I(L))^\N\big),
$
which exists by the fact that the cardinalities 
of the sets $C_I(L)$ and $\B(C_I(L))$ are $\kappa$ and
$\kappa^\omega=\kappa$ by Lemma \ref{cardinal} and Lemma \ref{density-cardinality}. 
By the standard closure argument and 
the fact that $\alpha^\omega<\kappa$ for all $\alpha<\kappa$ (Lemma \ref{cardinal}) the set
\[
  C_\Psi=\{\alpha\in\kappa: \Psi[\alpha]=(\{-2, -1\}\cup\alpha)\times \{-1, 1\}^\N \times  
  \big( C_I(L_\alpha)^\N \cup \B(C_I(L_\alpha))^\N\big)\}
\]
is a club set in $\kappa$.  While using $\Psi$ we will only be working with such subsets
$T\subseteq\alpha$ for $\alpha\in C_\Psi\cup\{\kappa\}$ that $\Psi[T]$ is a graph of
a function with two coordinate functions, first from $\{-2, -1\}\cup\alpha$ into $\{-1, 1\}^\N$
and the second from $\{-2, -1\}\cup\alpha$ into $ C_I(L_\alpha)^\N \cup \B(C_I(L_\alpha))^\N$
considered as a subset of $(\{-2, -1\}\cup\alpha)\times \{-1, 1\}^\N \times  
   \big(C_I(L_\alpha)^\N \cup \B(C_I(L_\alpha))^\N\big)$. That is, $\Psi$ will serve 
as a coding of such pairs of functions by subsets of ordinals in $\kappa$.
$\diamondsuit(E^\kappa_\omega)$ from Theorem \ref{diamond} will be our
prediction principle which for $\alpha\in C_\Psi$ may provide such a code for
the above pair of functions in the form of $T=S_\alpha\subseteq \alpha$.
\begin{theorem}\label{theorem3} Assume GCH. Let $\kappa$
be the successor of a cardinal of uncountable cofinality.
There is a compact Hausdorff connected c.c.c. space $K$ of weight
$\kappa$ without a butterfly point  such that $C(K)$ has asymmetric distribution of separations
in the direction of some $\mathcal D\subseteq C_I(K)$.
\end{theorem}
\begin{proof} We will construct a ladder  family $\F\subseteq C_I(L_\kappa)$ 
such that $K=\nabla\F$ satisfies the theorem.
Let $(S_\alpha)_{\alpha<\kappa}$ be a $\diamondsuit_\kappa(E^\kappa_\omega)$-sequence
as in \ref{diamond}.
Let $\Psi$ and $C_\Psi$ be as above.
For each
ordinal $\alpha\in E^\kappa_\omega$ choose a ladder  $(\eta_n^\alpha)_{n\in \N}$, that is 
  an increasing, cofinal in $\alpha$ sequence of type $\omega$.
The family $\mathcal F$ will depend of $S\subseteq \kappa$ and will be a ladder family
with the following parameters for $\alpha\in S$:
 \begin{itemize}
	\item $(\rho^\alpha_n)_{n\in\N} \subseteq \{-1,1\}$,
    \item  $(\eta^\alpha_n)_{n\in\N}$ ,
	\item  $(f^\alpha_n)_{n\in\N}\subseteq [\F]$, 
which depends on some $\beta_\alpha< \alpha$,
	\item  $M_\alpha\subseteq\N$.
  \end{itemize}
So we will use for it the terminology and
notation  as
in Definition \ref{ladder}. In fact the above parameters are build by
recursion  together with some additional objects which will witness the fact
that $\F$ has asymmetric distribution of separations. Namely, the recursive construction
involves:
\begin{enumerate}
\item $S=\{\alpha_\xi: \xi<\kappa\}\subseteq E^\kappa_\omega\cap C_\Psi$,
\item $\{\beta_{\alpha_\xi}: \xi<\kappa\}\subseteq \kappa$, $\beta_{\alpha_\xi}<\alpha_\xi$,
\item $(\rho^{\alpha_\xi}_n)_{n\in \N}\subseteq\{-1,1\}$,
\item $\{f_n^{\alpha_\xi}: n\in \N\}$ which is a pairwise disjoint sequence in the algebra $[\F[\beta_\xi]]$,
\item $\{V_n^{\xi}: n\in \N\}$ which is an antichain in  $\B(\F[\alpha_\xi])$ such that
 \[V_n^{\xi}(L)\cap \supp(f_m^{\alpha_\xi})=\emptyset,\ \hbox{for all $n, m\in\N$},\]
\item $A_\xi$ a countable subset of $\alpha_\xi$ such that $\F[A_\xi]$ is separating
in $\F[\alpha_\xi]$
for $\{V_n^{\xi}: n\in \N\}$ (see Definition \ref{separating}),
\item $M_{\alpha_\xi}\subseteq \N$  such that $\{V_n^{\xi}: n\in \N\}$ is
not separated in $[\F[A_\xi]]$ along $M_{\alpha_\xi}$,
\item $g_{\alpha_\xi}=\bigvee_{n\in M_{\alpha_\xi}}
(f_n^{\alpha_\xi}d_{\eta^{\alpha_\xi}_n, \rho^{\alpha_\xi}_n})$.
\end{enumerate}
Suppose that we have constructed all these objects for all $\xi<\gamma$ for
some $\gamma<\kappa$. This gives ladder families $\F[\alpha]$ 
for any $\sup\{\alpha_\xi:\xi<\gamma\}\leq\alpha\leq\kappa$, just consisting of 
the elements $\{g_{\alpha_\xi}:\xi<\gamma\}$  and $\{d_\beta: \beta<\alpha\}$ for $\alpha$ as above.
Let $\alpha_\gamma$
be the first ordinal in $E^\kappa_\omega\cap C_\Psi$ not smaller than
$\sup\{\alpha_\xi: \xi<\gamma\}$ such that 
\begin{enumerate}[(i)]
\item $\Psi[S_{\alpha_\gamma}]=
(\phi, \psi)$ is a function (identified with its graph)
from $(\{-2, -1\}\cup\alpha_\gamma)$ into $ \{-1, 1\}^\N
\times \big(C_I(L)^\N \cup \B(C_I(L))^\N\big)$.
\item $\psi(-2)$ is a a pairwise disjoint sequence $(f_n^{\alpha_\gamma})_{n\in \N}$ in $\F[\beta_\gamma]$
for some $\beta_\gamma<\alpha_\gamma$,
\item $\psi(-1)$ is an antichain 
$(V_n^\gamma)_{n\in \N}$ in $\B[\F[\alpha_\gamma]]$ 
such that for all $n, m\in \N$ we have 
$V^{\gamma}_n(L)\cap \supp(f_m^{\alpha_\gamma})=\emptyset$,
\item for $\alpha\in\alpha_\gamma$ the  value $\psi(\alpha)$ 
is an antichain $\{V_{\alpha, n}^\gamma: n\in \N\}$ in $\B[\F[\alpha_\gamma]]$ whose $n$-th 
element $V_{\alpha, n}^\gamma$ is 
below $V_n^\gamma$.
\end{enumerate}
 In this case we define 
\begin{enumerate}[(a)]
\item $\beta_\gamma$, $f_n^{\alpha_\gamma}$, $V_n^\gamma$ as above, 
 \item $\rho_n^{\alpha_\gamma}=\phi(\eta^{\alpha_\gamma}_n)(n)$,
\end{enumerate}
The existence of $\alpha_\gamma$ 
follows from $\diamondsuit(E^\kappa_\omega)$. 
To define $A_\gamma$ and $M_{\alpha_\gamma}$ we need to make some argument: 
The family $\F[\alpha_\gamma]$ is a ladder family and 
$\{V_{\eta^{\alpha_\gamma}_n, n}^{\gamma}: n\in \N\}$ is an antichain in $\B[\F[\alpha_\gamma]]$, hence
by Lemma \ref{separating-countable}, there is a countable $A_\gamma\subseteq \alpha_\gamma$
such that $\F[A_\gamma]$ is separating for$\{V_{\eta^{\alpha_\gamma}_n, n}^{\gamma}: n\in \N\}$ 
in $\F[\alpha_\gamma]$. Since $\F[A_\gamma]$ is a countable, by Lemma \ref{countable-family}
we can find an infinite $M_{\alpha_\gamma}\subseteq \N$ such that
\begin{enumerate}[(c)]
\item $\{V_{\eta^{\alpha_\gamma}_n, n}^{\gamma}: n\in \N\}$ is not separated in $[\F[A_\gamma]]$
along $M_{\alpha_\gamma}$.
\end{enumerate}
 Finally put
\begin{enumerate}[(d)]
\item $g_{\alpha_\gamma}=\bigvee_{n\in M_{\alpha_\gamma}}
(f_n^{\alpha_\gamma}d_{\eta^{\alpha_\gamma}_n, \rho^{\alpha_\gamma}_n})$
\end{enumerate}
This completes the inductive step in the construction of the ladder family $\mathcal F$.
Now let us prove that $C(\nabla\mathcal F)$ has asymmetric distribution of separations
in the direction of $\D=\{d_\alpha: \alpha<\kappa\}$. So, fix
  \begin{itemize}
	\item Pairwise disjoint $(f_n)_{n\in \N}\subseteq [\F]$ and and antichain $(U_n)_{n\in \N}\subseteq\B(\F)$
	  such that 
\[\supp(f_n)\cap {U_n(L)}=\emptyset \ \ \hbox{for all $n,m\in\N$},\]
	\item $(\nu_n^\xi)_{n\in\N} \subseteq \{-1, 1\}$ for all $\xi\in\kappa$,
	\item $\{\, (U_n^\xi)_{n \in \N} \mid \xi \in \kappa \,\}\subseteq \B(\F)$ 
	  such that $U_{n}^{\xi}(\F) \subseteq U_n(\F)$ for every $n\in\N$ and $\xi \in \kappa$.
  \end{itemize}
Let $X\subseteq \kappa$ be such that $\Psi[X]$ is 
a function 
$(\phi, \psi)$ from $\{-2, -1\}\cup\alpha_\gamma$ into $\{-1, 1\}^\N
\times \big(C_I(L)^\N \cup \B(C_I(L))^\N\big)$
(identified with its graph) such that:
\begin{itemize}
\item  $\phi(\alpha)(n)=\nu_n^\alpha$ for each $n\in \N$ and each $\alpha\in\kappa$,
\item $\psi(-2)$ is $(f_n)_{n\in \N}$,
\item $\psi(-1)$ is $(U_n)_{n\in \N}$,
\item $\psi(\alpha)$ is $(U_n^\alpha)_{n\in \N}$ for each $\alpha\in\kappa$.
\end{itemize}
By the properties of the  $\diamondsuit_\kappa(E^\kappa_\omega)$-sequence (Theorem \ref{diamond}), 
the facts that $E_\omega^\kappa$ is stationary and $C_\psi\setminus\beta$ is a club set, there is
$\alpha\in E^\kappa_\omega\cap [C_\psi\setminus(\beta+1)]$ such that $S_\alpha=X\cap\alpha$
where $\beta<\kappa$ is such that $f_n\in \F[\beta]$ and $U_n\in \B[\F[\beta]]$ for each $n\in\N$
which exists by the uncountable cofinality of $\kappa$. 
By the definition of $C_\Psi$ and the choice of $\alpha$ we have
\[\Psi[S_\alpha]=\Psi[X\cap\alpha]=\Psi[X]\cap \Psi[\alpha]=(\phi, \psi)\cap \Psi[\alpha]=(\phi,\psi)
\restriction(\{-2, -1\}\cup\alpha).\]
So, ({\it i})-({\it iv}) are satisfied, moreover, then $\alpha=\alpha_\gamma\in S$ for some $\gamma<\kappa$.
In particular, by the construction (1) - (8) we have
\begin{itemize}
\item $\rho^{\alpha_\gamma}_n=\nu_n^{\eta_n}$, where $\eta_n={\eta_n^{\alpha_\gamma}}$,
\item $f_n^{\alpha_\gamma}=f_n$ for each $n\in \N$,
\item $\beta_\gamma=\beta$,
\item $V_n^{\gamma}=U_{n}$  for each $n\in \N$.
\item $V_{\alpha, n}^{\gamma}=U_{n}^\alpha$ for all $\alpha<\alpha_\gamma$ and
each $n\in\N$.
\end{itemize} 
So  
\[g_{\alpha_\gamma}=\bigvee_{n\in M_{\alpha_\gamma}}
(f_n^{\alpha_\gamma}d_{\eta^{\alpha_\gamma}_n, \rho^{\alpha_\gamma}_n})=
\bigvee_{n\in M_{\alpha_\gamma}}(f_nd_{\eta^{\alpha_\gamma}_n, \nu^{\eta_n}_n})\]
as required in Definition \ref{paracomplicated}.
So it remains to prove
that 
 the antichain $(U_{n}^{\eta_n})_{n\in \N}$ is not separated along the set
 $M_\gamma$ in $[\F]$.
First note that $(U_{n}^{\eta_n})_{n\in \N}$ is $(V_{\eta_n, n}^{\gamma})_{n\in \N}$
so is  not separated along $M_\gamma$ 
in $[\F[\alpha_\gamma]]$ by (c).  Now, since 
$U_n\cap \supp(f_m)=\emptyset$ for all $n, m\in \N$,  and $U_n^\alpha\subseteq
U_n$ by Lemma \ref{adding-disjoint} and Lemma \ref{separating-ordinal}
we conclude that  $(U_{n}^{\eta_n})_{n\in \N}$ is not separated 
along $M_\gamma$ in $[\F]$. 
So,  $C(\nabla\F)$ has asymmetric distribution of separations in
the direction of $\{d_\alpha: \alpha<\kappa\}$.
$\nabla\Delta$ is connected and has no butterfly points by
Lemma \ref{ladder2-nobutterfly} and Lemma \ref{ladder-connected}.
It is c.c.c. as a continuous image of a c.c.c. space $L_\kappa$ by Proposition \ref{gleason},
and has weight $\kappa$ by Lemma \ref{density-cardinality},
so the proof is completed.
\end{proof}
\bibliographystyle{amsplain}

\end{document}